\documentclass[12pt]{amsart}

\usepackage[hypertex]{hyperref}
\usepackage{amsmath}
\usepackage{mathtools}
\usepackage{amssymb}
\usepackage{amsfonts}
\usepackage{amsopn}
\usepackage{amsthm}
\usepackage[all]{xy}
\usepackage{color}
\usepackage{longtable}
\usepackage{verbatim}

\swapnumbers
\theoremstyle{plain}
\newtheorem{thm}{Theorem}[section]

\newtheorem{lem}[thm]{Lemma}
\newtheorem{cor}[thm]{Corollary}
\newtheorem{prop}[thm]{Proposition}

\theoremstyle{definition}
\newtheorem{ntt}[thm]{}

\newtheorem{ex}[thm]{Example}
\newtheorem{rem}[thm]{Remark}
\newtheorem{dfn}[thm]{Definition}

\newcommand{\zz}{\mathbb{Z}} 
\newcommand{\DM}{\mathrm{DM}} 

\newcommand{\qq}{\mathbb{Q}} 
\newcommand{\rr}{\mathbb{R}}

\newcommand{\laz}{\mathbb{L}}
\newcommand{\pp}{\mathbb{P}}
\newcommand{\id}{\mathrm{id}}
\newcommand{\lla}{\langle\!\langle}
\newcommand{\rra}{\rangle\!\rangle}

\newcommand{\D}{\mathrm{D}}
\newcommand{\E}{\mathrm{E}}
\newcommand{\F}{\mathrm{F}_4}
\newcommand{\llarrow}{\mathrel{\vcenter{\vbox{\offinterlineskip
\hbox{$\leftarrow$}\hbox{$\leftarrow$}}}}}
\newcommand{\lllarrow}{\mathrel{\vcenter{\vbox{\offinterlineskip
\hbox{$\leftarrow$}\hbox{$\leftarrow$}\hbox{$\leftarrow$}}}}}

\DeclareMathOperator{\Hom}{\mathrm{Hom}}
\DeclareMathOperator{\End}{\mathrm{End}}
\DeclareMathOperator{\Cone}{\mathrm{Cone}}
\DeclareMathOperator{\PGL}{\mathrm{PGL}}

\newcommand{\pr}{\mathrm{pr}}
\newcommand{\res}{\mathrm{res}}

\newcommand{\BX}{\overline{X}}

\newcommand{\pt}{\mathrm{pt}}

\DeclareMathOperator{\Spec}{\mathrm{Spec}}
\DeclareMathOperator{\CH}{\mathrm{CH}}
\DeclareMathOperator{\Ch}{\mathrm{Ch}}

\title{Motivic construction of cohomological invariants}
\author{NIKITA SEMENOV}
\thanks{The author gratefully acknowledges the support of the Sonderforschungsbereich/Transregio 45 (Bonn-Essen-Mainz).}

\date{}

\begin{document}
\maketitle
\begin{abstract}
Let $G$ be a group of type $\E_8$ over $\qq$ such that $G_\rr$ is
a compact Lie group, let $K$ be a field of characteristic $0$,
and $$q=\langle\!\langle -1,-1,-1,-1,-1\rangle\!\rangle$$ a $5$-fold Pfister form.
J-P.~Serre posed in a letter to M.~Rost written on June 23, 1999
the following problem: Is it true that
$G_K$ is split if and only if $q_K$ is hyperbolic?

In the present article we construct a cohomological invariant of degree $5$ for groups
of type $\E_8$ with trivial Rost invariant over any field $k$ of characteristic $0$,
and putting $k=\mathbb{Q}$ answer positively this
question of Serre. Aside from that, we show that a variety which possesses
a special correspondence of Rost is a norm variety.
\end{abstract}

\section{Introduction}
Let $G$ be a group of type $\E_8$ over $\qq$ such that $G_\rr$ is
a compact Lie group.
Let now $K/\qq$ be a field extension
and $q=\langle\!\langle -1,-1,-1,-1,-1\rangle\!\rangle$
a $5$-fold Pfister form.
J-P.~Serre posed in a letter to M.~Rost written on June 23, 1999
the following problem:

\medskip

\centerline {Is it true that
$G_K$ is split if and only if $q_K$ is hyperbolic?}

\medskip

M.~Rost replied on July 2, 1999 proving that if $q_K$ is hyperbolic,
then $G_K$ is split (see \cite{GaSe09} for the proof).
One of the goals of the present article is to give a positive answer to Serre's question
(see Theorem~\ref{maintheorem}).

Let us recall first some recent developments in the topics which are relevant for
the method of the proof of this result.

\subsection{Cohomological invariants}
The study of cohomological invariants was initiated
by J-P.~Serre in the 90'ies. Serre conjectured the existence of an invariant
in $H_{et}^3(k,\qq/\zz(2))$
of $G$-torsors, where $G$ is a simply connected simple
algebraic group over a field $k$.
This invariant was constructed by M.~Rost and
is now called the {\it Rost invariant} of $G$
(see \cite[31.B and pp.~448--449]{Inv}).

Nowadays there exist numerous constructions and estimations of cohomological
invariants for different classes of algebraic objects (see e.g. \cite{GMS}).
Nevertheless, the most constructions of cohomological invariants rely
on a specific construction of the object under consideration.
Unfortunately, for many groups, like $\E_8$,
there is no classification and no general construction so far.

\subsection{Chow motives of twisted flag varieties}
Another direction which we discuss now is the theory of motives
of twisted flag varieties. Typical examples of such varieties
are projective quadrics and Severi--Brauer varieties.

The study of twisted flag varieties was initiated by
Rost when he provided a motivic decomposition of a Pfister quadric
used later in the proof of the Milnor conjecture (see \cite{Ro98}).
The motives of Severi--Brauer varieties were studied by Karpenko and the motives
of general quadrics by Izhboldin, Karpenko, Merkurjev, and Vishik.

A systematical theory of motives of general twisted flag varieties
was developed in a series of our articles with a culmination in \cite{PSZ}
and \cite{PS} where the structure of the motive of a generically split twisted flag variety
was established. Moreover, in some cases we provided an alternative construction of {\it generalized
Rost motives} as indecomposable direct summands of generically split varieties
(see \cite[\S7]{PSZ}).

The technique developed in our articles does not use cohomological
invariants of algebraic groups and of twisted flag varieties at all,
and our arguments do not rely on specific constructions of algebraic groups.

\subsection{The Bloch--Kato Conjecture}
The Bloch--Kato conjecture (more precisely its proof) provides a bridge
between cohomological invariants and motives.

The Bloch--Kato conjecture says that for any $n$ and any prime number $p$
the norm residue homomorphism
$$K_n^M(k)/p\to H_{et}^n(k,\mu_p^{\otimes n})$$
$$\{a_1,\ldots,a_n\}\mapsto(a_1)\cup\ldots\cup(a_n)$$
between the Milnor K-theory and the Galois cohomology of a field $k$ with $\mathrm{char}\,k\ne p$
is an isomorphism.

The case $p=2$ of the Bloch--Kato conjecture is known as the Milnor conjecture.

The proof of the Bloch--Kato conjecture goes as follows. Given a pure symbol
$u$ in $H_{et}^n(k,\mu_p^{\otimes n})$ M.~Rost constructs a splitting variety $X_u$
for $u$ with some additional properties (If $p=2$ one takes norm quadrics, 
for $n=2$ one takes Severi--Brauer varieties, for
$p=3$ the construction is based on the Merkurjev--Suslin varieties;
see \cite{S08} for an explicit construction for $p=n=3$).
Then V.~Voevodsky using the category of his motives and symmetric powers
splits off a certain direct summand $M_u$ from the motive
of $X_u$ (see \cite{Vo03}). The motive $M_u$ is called the {\it generalized Rost motive}.

Later Rost was able to simplify Voevodsky's construction and found the same motive
$M_u$ staying inside the category of classical Chow motives (see \cite{Ro07}).
His main idea
was to produce an algebraic cycle on $X_u$ called a {\it special correspondence}
(see Definition~\ref{special}),
forget about the symbol $u$ and use the special correspondence to construct
the motive $M_u$.
This motive is used then to construct exact triangles in the category
of Voevodsky's motives. These triangles essentially involve mixed motives,
in particular, deviate from the category of classical Chow motives,
and are used (among other results of Rost and Voevodsky) to finish the
proof of the Bloch--Kato conjecture.

The symbol $u$ is a cohomological
invariant of the variety $X_u$ (see \cite[Theorem~6.19]{Vo03}).
Aside from that, one can notice
that the special correspondences of Rost resemble the algebraic cycles that
we construct in the proof of our motivic decompositions
(compare \cite[Lemma~5.2]{Ro07} and \cite[Lemma~5.7]{PSZ}).

Summarizing: The technique developed by Rost and Voevodsky in the proof
of the Bloch--Kato conjecture gives a way to produce algebraic cycles and motives
out of cohomological invariants of varieties.

\subsection{Conclusion}
The discussion in the above subsections leads to the following conjecture.
Namely, there should exist a way back from
motives to cohomological invariants. 
Notice that an evidence of this conjecture appears already in articles of
O.~Izhboldin and A.~Vishik \cite{IzhV00} and \cite{Vi98}, where the case of
quadrics was treated.

In the present article we show in a constructive way that the varieties
which possess a special correspondence of Rost admit cohomological
invariants.

This result has several unexpected applications in the theory of algebraic groups.
In the present article we use it to construct
a cohomological invariant in $H_{et}^5(k,\zz/2)$ for any group of type $\E_8$
whose Rost invariant is trivial. In turn, this invariant allows us to give
in the last section a positive answer to a question of Serre about compact groups of type
$\E_8$ mentioned above.

We remark that the construction of this invariant essentially
relies on the motivic decomposition of the respective variety of Borel subgroups.
We compute this decomposition in Section~\ref{sec8} using the $J$-invariant of algebraic groups
(see \cite{PSZ}) and the classification of generically split twisted flag varieties (see \cite{PS} and \cite{PS11}).

\subsection{\textbf{\emph{Leitfaden} of the proof}}
Let us present a simplified scheme of the construction of the degree $5$ invariant for groups
of type $\E_8$.

Let $G$ be an anisotropic group of type $\E_8$ over a field $k$ of characteristic $0$ with trivial Rost invariant and
let $X$ be the variety of Borel subgroups of $G$.

Using ``compression'' described in Section~\ref{seccompr} we construct a smooth projective variety $\widetilde Y$
over $k$ of dimension $15=2^{5-1}-1$ with no $0$-cycles of odd degree and such that the direct summands of
the motives of $X$ and $\widetilde Y$ supporting the $0$-cycles are isomorphic (see Definition~\ref{uppermot}).

Next we show in Lemma~\ref{lemmae8}:
\begin{center}   
\fbox{$J$-invariant of algebraic groups (\cite{PSZ}, \cite{PS}, and \cite{PS11})}
\end{center}
$$\Downarrow \text{{\footnotesize{the Rost invariant of {\it G} is zero}}}$$
\begin{center}
\fbox{
$\begin{array}{l}
\text{Direct summand of the motive of $X$ (and hence of $\widetilde Y$) supporting}\\
\text{the $0$-cycles is {\bf binary}}
\end{array}
$}
\end{center}

\medskip

\newpage

Finally, we show (Theorem~\ref{binmot}):

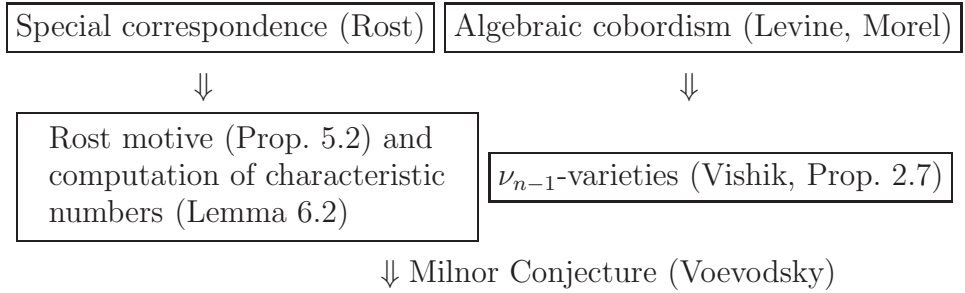
\medskip

\begin{center}
\fbox{Special correspondence (Rost)}~\fbox{Algebraic cobordism (Levine, Morel)}
\end{center}

\smallskip

$\quad\quad\quad\quad\quad\quad\quad\quad\quad\Downarrow\quad\quad\quad\quad\quad\quad\quad\quad\quad\quad\quad\quad\quad\quad\quad\Downarrow$
\begin{center}
\fbox{
$\begin{array}{l}
\text{Rost motive (Prop.~\ref{rostmotive}) and}\\
\text{computation of characteristic}\\
\text{numbers (Lemma~\ref{rostbinary})}
\end{array}$
}  \fbox{$\nu_{n-1}$-varieties (Vishik, Prop.~\ref{voeconj})}
\end{center}
$$\qquad\qquad\qquad\qquad\Downarrow\text{Milnor Conjecture (Voevodsky)}$$
\begin{center}
\fbox{
$\begin{array}{c}
\text{A smooth projective variety $\widetilde Y$ over $k$ of dimension $2^{n-1}-1$ with} \\
\text{no $0$-cycles of odd degree and with a binary motivic summand supporting the}\\
\text{$0$-cycles possesses a cohomological invariant $u\in H_{et}^n(k,\zz/2)$}\\
\text{such that for any field extension $K/k$ one has $u_K=0$ iff $\widetilde Y_K$ has a $0$-cycle of odd degree}
\end{array}$
}
\end{center}

\medskip

All this together gives an invariant $u$ of degree $5$ for $\E_8$.

We remark that the number $15=\dim\widetilde Y$ has a combinatorial origin. Namely, it is related to degrees
of certain polynomial invariants of the Weyl group of type $\E_8$ (see \cite{Ka85}).

Let now $G$ be a group of type $\E_8$ over $\qq$ of compact type and let $G_0$ be a split group of type $\E_8$ over
a field $K$ of characteristic zero. Then $u=(-1)^5\in H_{et}^5(\qq,\zz/2)$ and one can show (see \cite{GaSe09}):

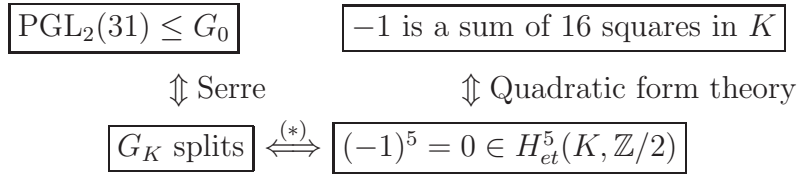
\begin{center}
\fbox{$\mathrm{PGL}_2(31)\le G_0$} \quad\quad\quad  \fbox{$-1$ is a sum of $16$ squares in $K$}
\end{center}

\smallskip

$\quad\quad\quad\quad\quad\quad\quad\quad\quad\quad\quad\Updownarrow \text{Serre}\quad\quad\quad\quad\quad\quad\Updownarrow\text{Quadratic form theory}$
\begin{center}                                                                                      
\fbox{$G_K$ splits} $\overset{(*)}{\Longleftrightarrow}$ \fbox{$(-1)^5=0\in H_{et}^5(K,\mathbb{Z}/2)$}
\end{center}

Thus, the positive solution to Serre's problem $(*)$ implies some classification results about finite subgroups of algebraic groups
over a field $K$. Moreover, instead of $16$ squares one can take any number of squares between $16$ and $31$ (see \cite{Pf65}).
We remark finally that one can write $31$ in $\PGL_2(31)$ as $30+1$, and $30$ is the Coxeter number for $\E_8$,
see \cite[Section~2.1, Example 5]{Se00} for a general case.

\subsection{Structure of the article}
In section~\ref{cobordism} we recall some results on algebraic cobordism of Levine--Morel and
$\nu_n$-varieties. We need them to prove a certain statement about $\nu_n$-varieties (Prop.~\ref{voeconj})
which appears as conjecture~1 before Thm.~6.3 in \cite{Vo03}. The proof
belongs to A.~Vishik, and we represent it here with his kind permission.
In Sections~\ref{sec3} and \ref{skeleton} we recall basic properties of the category of motives of Voevodsky
and the Rost construction of special correspondences resp.
Sections~\ref{mainsection} and \ref{sec6} are devoted to the construction of cohomological invariants for
varieties admitting a special correspondence (Theorems~\ref{mainthm} and \ref{binmot}).
In section~\ref{seccompr} we provide some general results which allow to ``compress''
varieties. Finally, in section~\ref{sec8} we construct an invariant for groups of type $\E_8$
with trivial Rost invariant and solve Serre's problem (Theorem~\ref{maintheorem}) finishing the proofs of all main results of the present article.

\subsection{Acknowledgments}
The first version of this article appeared on arxiv.org in May, 2009.
I would like to thank sincerely Skip Garibaldi, Nikita Karpenko, Alexander Merkurjev,
Fabien Morel, Victor Petrov, Alexander Vishik, Kirill
Zainoulline, Maksim Zhykhovich, and especially Stefan Gille for encouragement and for interesting discussions and remarks
on the subject of the article over a long period of time.

\section{Algebraic cobordism and $\nu_n$-varieties}\label{cobordism}
\begin{ntt}[Lazard ring]
Fix a prime $p$ and a field $k$ with $\mathrm{char}\,k=0$.
Consider the algebraic cobordism $\Omega^*$ of Levine--Morel (see \cite{LM})
and the Lazard ring $\laz=\Omega(\Spec k)$.
This is a graded ring additively generated by the
cobordism classes of maps $$[X]=[X\to\Spec k]\in\laz^{-\dim X}$$
where $X$ is an irreducible smooth projective variety over $k$.
It is known that this ring is isomorphic to the polynomial ring with
integer coefficients on a countable set of variables.
\end{ntt}

\begin{ntt}[Characteristic numbers]\label{charnum}
Given a partition
$J=(l_1,\ldots,l_r)$ of arbitrary length $r\ge 0$ with $l_1\ge l_2\ge\ldots\ge l_r>0$
one can associate with it a {\it characteristic class}
$$c_J(X)\in\CH^{|J|}(X)\qquad (|J|=\sum_{i\ge 1} l_i)$$ of $X$ as follows:
Let $P_J(x_1,\ldots,x_r)$ be the smallest symmetric polynomial (i.e., with a minimal number
of non-zero coefficients) containing
the monomial $x_1^{l_1}\ldots x_r^{l_r}$.
We can express $P_J$ as a polynomial on the standard symmetric functions
$\sigma_1,\ldots,\sigma_r$ as
$$P_J(x_1,\ldots,x_r)=Q_J(\sigma_1,\ldots,\sigma_r)$$ for some polynomial
$Q_J$. Let $c_i=c_i(-T_X)$ denote the $i$-th Chern class 
of the virtual normal bundle of $X$. Then
$$c_J(X)=Q_J(c_1,\ldots,c_r).$$ The degrees of the characteristic classes
are called the {\it characteristic numbers}.

If $J=(1,\ldots,1)$ ($i$ times), $c_J(X)=c_i(-T_X)$
is the usual Chern class. We denote
$b_{i(p-1)}(X):=c_{(p-1,\ldots,p-1)}(X)$ ($i$ times).
In particular, if $p=2$ then $b_i(X)=c_i(-T_X)$.
If $J=(p^n-1)$, we write $c_J(X)=s_{p^n-1}(X)$.
The class $s_d(X)$ is called the $d$-th {\it Milnor class} of $X$
(see \cite[Section~4.4.4]{LM}).

The characteristic numbers satisfy different divisibility properties. Moreover, there exist
relations between different characteristic numbers (see e.g. \cite[Section~9]{Ro07}, \cite[Prop.~7.11]{Me02}).
\end{ntt}

\begin{dfn}[$\nu_n$-varieties]
Let $p$ be a fixed prime.
\begin{enumerate}
\item[(I)] A smooth projective variety $X$ is called a $\nu_n$-variety if
$\dim X=p^n-1$, all characteristic numbers of $X$ are divisible by $p$, and $$\deg s_{\dim X}(X)\ne 0\text{ mod }p^2.$$
\item[(II)] A smooth projective variety $X$ is called a $\nu_{\le n}$-variety
if $X$ is a $\nu_n$-variety and
for all $0\le i\le n$ there exists a $\nu_i$-variety $Y_i$ and a morphism
$Y_i\to X$.
\end{enumerate}
\end{dfn}

\begin{ex}\label{cobexa}
A typical example of a $\nu_n$-variety is a smooth
projective hypersurface $u$ of degree $p$ in $\mathbb{P}^{p^n}$ (see \cite[Proposition~3.6]{Vo01}).

The fact that all characteristic classes of $u$ are divisible by $p$ follows
from the divisibility of its characteristic numbers by $\deg h^{p^n-1}=p$, where $h$ is a hyperplane section
of $u$.
\end{ex}

\begin{rem}
The above definition of a $\nu_n$-variety is more restrictive than the definition which
Voevodsky uses in his articles. Namely, he does not assume that all characteristic numbers
are divisible by $p$.
\end{rem}

By \cite[Lemma~4.4.19]{LM} one can extend the definition of characteristic numbers to an additive
map $\Omega_d(\Spec k)\to\zz$, where $d=\deg P_J(x_1,\ldots,x_r)$.
In particular, one may speak about a {\it $\nu_n$-element} in the Lazard ring, i.e., an element $u$ in $\laz_{p^n-1}$
with all characteristic numbers divisible by $p$ and with $\deg s_{p^n-1}(u)\ne 0\mod p^2$.

\begin{ntt}[Landweber--Novikov operations]
Further on one can construct the Landweber--Novikov operations on
the cobordism ring $\Omega^*(V)$ for a smooth quasi-projective variety $V$.
Given a partition $J$ as above
we write $$S_{L-N}^J\colon\Omega^*(V)\to\Omega^{*+|J|}(V)$$ for the
operation which maps a cobordism class $[f\colon X\to V]\in\Omega^*(V)$
to $f_*(Q_J(c^\Omega_1,\ldots,c^\Omega_r))$, where
$c^\Omega_i=c^\Omega_i(-T_X+f^*(T_V))$
are the cobordism Chern classes (see \cite[Definition~2.12]{Vi07}).

Since $\Omega^*$ is a universal cohomology theory, there exists
a canonical map $\pr\colon\Omega^*\to\CH^*$ which turns out to be surjective.
For a particular partition $J=(p-1,\ldots,p-1)$ ($i$ times) denote
$S_{L-N}^i:=S_{L-N}^J$. Then the following diagram commutes:
$$
\xymatrix{
\Omega^*(V) \ar[d]_{S_{L-N}^i}\ar[r]^-{\pr} & \CH^*(V)\ar[r] & \Ch^*(V)\ar[d]^{S^i}\\
\Omega^{*+(p-1)i}(V)\ar[r]^-{\pr} & \CH^{*+(p-1)i}(V)\ar[r] & \Ch^{*+(p-1)i}(V) }
$$
where $\Ch^*:=\CH^*/p$ and $S^i$ denotes the $i$-th Steenrod operation
in the Chow theory (see \cite{Br03}).
In particular, since $S^i\vert_{\Ch^m(V)}=0$ for $i>m$, 
the image $$\pr\circ S^i_{L-N}(\Omega^m(V))$$ is divisible by $p$ for such $m$.
Moreover, the total operation $S_{L-N}:=\sum_iS^i_{L-N}$ is a ring
homomorphism (see \cite{Za08}, \cite{Vi07}, and \cite{LM}).

For an integer $n$ let $I(p,n)$ denote the ideal of $\laz$ 
generated by the varieties of dimension $\le p^n-1$
with all characteristic numbers divisible by $p$.
This is a prime ideal invariant under the action of the
Landweber--Novikov operations on $\laz$.
By definition $I(p,n-1)\subset I(p,n)$.
Moreover, as an $\laz$-ideal $I(p,n)$ is generated by $I(p,n-1)$
and by any $\nu_n$-variety. Besides, $I(p,n-1)$ does not contain $\nu_n$-varieties
(see \cite[Section~11]{Me02}, \cite[Section~3]{ViYa07}, and \cite[Section~2]{La73}).
\end{ntt}

The following proposition seems to be known in algebraic topology.
The existence of Landweber--Novikov operations in the algebraic cobordism theory
of Levine--Morel allows to prove it in the context of algebraic geometry.
This proposition appears as conjecture~1 before Thm.~6.3 in \cite{Vo03}.

\begin{prop}\label{voeconj}
A $\nu_n$-variety is a $\nu_{\le n}$-variety.
\end{prop}
\begin{proof}
(A.~Vishik)
It is well known that a smooth hypersurface of degree $p$ in
$\pp^{p^n}$ is a $\nu_n$-variety (see \cite[Prop.~3.6]{Vo01} and Example~\ref{cobexa}).
Denote its class in $\laz$ as $u$.

There are the following exact sequences of vector bundles
$$0\to T_u\to \iota^*(T_{\mathbb{P}^{p^n}})\to \iota^*(\mathcal{O}(p))\to 0\text{ and}$$
$$0\to\mathcal{O}\to \mathcal{O}(1)^{p^n+1}\to T_{\mathbb{P}^{p^n}}\to 0,$$
where $T_u$ is the tangent bundle of $u$ and $\iota\colon u\hookrightarrow\mathbb{P}^{p^n}$.

Therefore $c^{\Omega}(-T_u)=\dfrac{1+[p]h}{(1+h)^{p^n+1}}$, where $h$ is a hyperplane section of $u$
and $$[p]h=h+_\Omega +\ldots+_\Omega h=ph+\text{decomposable terms}$$ is the sum with respect to the cobordism formal
group law. A direct computation
shows now that $S_{L-N}^{p^{n-1}}(u)$ modulo $I(p,n-2)$ is a multiple of the class of
a hypersurface of degree $p$ in $\pp^{p^{n-1}}$ with a coefficient not
divisible by $p$. In particular, $S_{L-N}^{p^{n-1}}(u)$ is a
$\nu_{n-1}$-element.

Let now $v_n$ be a $\nu_n$-element in $\laz$. Then
$y:=av_n-u\in I(p,n-1)$ for some $a\in\zz$ coprime to $p$.
To conclude that $v_n$ is a $\nu_{n-1}$-element, it suffices
to show that $S_{L-N}^{p^{n-1}}(y)\in I(p,n-2)$.
We can write $$y=v_{n-1}x_{n-1}+v_{n-2}x_{n-2}+\ldots+v_0x_0$$
where $v_i$ is a $\nu_i$-element and $x_i\in\laz$.
Since the ideal $I(p,n-2)$ is stable under the Landweber--Novikov operations, we have that
$$S_{L-N}^{p^{n-1}}(v_{n-2}x_{n-2}+\ldots+v_0x_0)\in I(p,n-2).$$

Moreover, $S_{L-N}^i(v_{n-1})\in I(p,n-2)$ for all $i>0$.
Finally, we have the term $$S_{L-N}^0(v_{n-1})S_{L-N}^{p^{n-1}}(x_{n-1})=
v_{n-1}S_{L-N}^{p^{n-1}}(x_{n-1}).$$
But $S_{L-N}^{p^{n-1}}(x_{n-1})$ is divisible by $p$, since $x_{n-1}\in\laz^0$ and
this Landweber--Novikov operation induces the Steenrod operation
$S^{p^{n-1}}$ on the Chow group.
Thus, $$S_{L-N}^{p^{n-1}}(y)\in I(p,n-2),$$ and we are done.
\end{proof}

\section{Category DM}\label{sec3}
The base field $k$ is assumed to have characteristic zero, and
all varieties are assumed to be irreducible.

If $A\to B$ is a morphism in a triangulated category, then there
exists a unique up to isomorphism object $C$ with an exact triangle
of the form $A\to B\to C$. We write $C=\Cone(A\to B)$.

\begin{dfn}
Let $X$ be a smooth projective variety over $k$. As $\check C(X)$ we denote the {\it standard simplicial scheme}
associated with $X$:
$$X\llarrow X\times X\lllarrow X\times X\times X\cdots$$
\end{dfn}

In the present article we work in the category $\DM:=\DM^{eff}_{-}(k)$
of effective motives of V.~Voevodsky over $k$ with $p$-adic coefficients $\zz_p$,
where $p$ is a fixed prime number (see \cite{MVW06}, \cite{Vo95}, \cite{Vo03}).
In the same way we could work with $\zz_{(p)}$-coefficients, but this would make the proofs
of some statements (e.g. of Lemma~\ref{lift}) technically more involved.

It is known that there
are functors from the category of smooth schemes over $k$ and from the
category of smooth simplicial schemes over $k$ to $\DM$. Therefore we can speak
about the motives of smooth schemes and smooth simplicial schemes.

The category of classical Grothendieck's Chow motives is a full subcategory
of $\DM$ closed under direct summands. Its objects are called {\it pure motives}.
For an integer $b$ we set $\{b\}:=(b)[2b]$. For example, if $M\in\DM$,
then $M\{b\}[1]$ means $M(b)[2b+1]$ and $M\{b\}$ is a {\it Tate twist}
of $M$. We have an exact triangle:
\begin{align}\label{zp}
\zz_{p}\xrightarrow{p}\zz_{p}\to\zz/p\to\zz_{p}[1],
\end{align}
where $\zz_{p}$ is the motive of the base point.
For $M\in\DM$ its {\it motivic
cohomology} with $\zz/p$-coefficients is given by $$H^{d,c}(M,\zz/p)=\Hom(M,\zz/p(c)[d]).$$

We list now some properties of the category $\DM$.

\begin{lem}[{\cite[Cor.~19.2 and Thm.~5.1]{MVW06}}]\label{puremotive}
Let $X$ be a smooth projective variety over $k$.
Then
$$\Hom(X,\zz/p(i)[2i])=\CH^i(X)\otimes\zz/p=:\Ch^i(X)$$ is the Chow group of $X$.
If $X=\Spec k$, then additionally $$\Hom(X,\zz/p(i)[i])=K^M_i(k)/p$$
is the Milnor $K$-theory of $k$ mod $p$.
\end{lem}

We remark that these statements are given in \cite{MVW06} integrally. The mod-$p$ version follows
from the integral version, from exact triangle~\eqref{zp} and from the facts that
$$\Hom(X,\zz_{p}(i)[2i+1])=0 \text{ and }\Hom(\Spec k,\zz_{p}(i)[i+1])=0$$ for all $i\ge 0$
(see \cite[Thm.~3.6 and Thm.~19.3]{MVW06}).

\begin{lem}[{\cite[Theorem~2.3.4]{Vi98}, \cite[App.~B]{Vo01}}]\label{lemma33}
Let $X$ and $Y$ be smooth projective irreducible varieties over $k$ and let $\mathcal{X}_X$ (resp. $\mathcal{X}_Y$) be
the motive of $\check C(X)$ (resp. of $\check C(Y)$). The natural morphism
$\mathcal{X}_X\to\Spec k$ is an isomorphism if and only if $X$ has a zero-cycle
of degree coprime to $p$.

The motives $\mathcal{X}_X$ and $\mathcal{X}_Y$ are isomorphic if and only if
$X$ has a zero-cycle of degree coprime to $p$ over $k(Y)$ and $Y$ has a zero-cycle of degree coprime to $p$ over $k(X)$.
\end{lem}

\begin{lem}[{\cite[Thm.~3.6 and Thm.~19.3]{MVW06}}]\label{zerohom}
Let $X$ be a smooth projective variety over $k$.
Then $$\Hom(X,\zz/p(c)[d])=0$$ for $d-c>\dim X$ and for $d>2c$. The same formula holds for motivic
cohomology with integral coefficients.
\end{lem}
We remark that again this mod-$p$ version follows from the integral version given in \cite{MVW06} and from
exact triangle~\eqref{zp}.

\begin{lem}[Cancellation; {\cite[Corollary~4.10]{Vo02}}]\label{cancellation}
Let $A,B\in\DM$.
Then for all $i,j\ge 0$, we have $$\Hom(A,B)=\Hom(A(i)[j],B(i)[j]).$$
\end{lem}

\begin{dfn}
A localizing subcategory of $\DM$ generated by an object
$A\in\DM$ is a minimal triangulated subcategory closed under direct summands, arbitrary direct sums,
and containing $A$.
\end{dfn}

\begin{lem}[{\cite[Lemma~4.7]{Vo01}, \cite[Theorem~2.3.2]{Vi98}}]\label{localize}
Let $X$ be a smooth projective variety over $k$ and $M$ an object of
the localizing subcategory of $\DM$ generated by $X$. Denote
as $\mathcal{X}$ the motive of $\check C(X)$. Then
\begin{enumerate}
\item[(I)] the natural morphism $M\otimes\mathcal{X}\to M$ is an isomorphism,
\item[(II)] and the natural homomorphism $\Hom(M,\mathcal{X})\to\Hom(M,\zz_{p})$ is an isomorphism.
\end{enumerate}
\end{lem}

In $\DM$ there is an action of the {\it Milnor operations} 
for smooth simplicial schemes $U$
$$Q_i\colon\Hom(U,\zz/p(r)[s])\to\Hom(U,\zz/p(r+p^i-1)[s+2p^i-1]).$$
In the same way there is an action of the Milnor operations on the
{\it reduced motivic cohomology} $\widetilde H^{*',*}(U,\zz/p)$
for smooth pointed simplicial schemes $U$. They
satisfy the following properties:

\begin{lem}[{\cite[Section~13]{Vo03r} and \cite[Thm.~3.17]{Vo96}}]\label{milnormulti}
Let $U$ be a smooth pointed simplicial scheme $u,v\in\widetilde H^{*',*}(U,\zz/p)$.
Then 
\begin{enumerate}
\item[(I)] $Q_i^2=0$;
\item[(II)] $Q_iQ_j+Q_jQ_i=0$;
\item[(III)] $Q_i(uv)=Q_i(u)v+uQ_i(v)+\sum\xi^{n_j}\varphi_j(u)\psi_j(v)$,
where $n_j>0$, $\varphi_j$ and $\psi_j$ are some homogeneous
cohomological operations of bidegrees $(*',*)$ with $*'>2*\ge 0$, and $\xi$ is the class
of $-1$ in $\Hom(\Spec k,\zz/p(1)[1])=k^\times/k^{\times p}$. (In particular,
$\xi=0$ if $p$ is odd).
\end{enumerate}
\end{lem}

Thus, the following sequence is a complex:
\begin{equation}\label{margolis}
\widetilde H^{*-2p^i+1,*'-p^i+1}(U,\zz/p)\overset{Q_i}{\to}\widetilde H^{*,*'}(U,\zz/p)
\overset{Q_i}{\to}\widetilde H^{*+2p^i-1,*'+p^i-1}(U,\zz/p).
\end{equation}

\begin{dfn}
The cohomology groups of complex~\eqref{margolis} are called {\it Margolis motivic
cohomology groups} and are denoted as $\widetilde{HM}_i^{*,*'}(U)$.
\end{dfn}

The following lemma is of particular importance for us.
\begin{lem}\label{margolislemma}
Let $X$ be a smooth projective variety over $k$, $Y$ a $\nu_n$-variety over $k$,
and let $\mathcal{X}$ be the motive of $\check C(X)$. Denote as
$\widetilde{\mathcal{X}}=\Cone(\mathcal{X}\to\Spec k)$ the cone
of the natural projection. 

Assume that there exists a map $Y\to X$.
Then for all $0\le i\le n$ and all $*,*'\in\zz$
$$\widetilde{HM}_i^{*,*'}(\widetilde{\mathcal{X}})=0.$$
\end{lem}
\begin{proof}
The lemma immediately follows from \cite[Lemma~4.3]{Vo03} and Prop.~\ref{voeconj}.
\end{proof}

The following lemma immediately follows from exact triangle~\eqref{zp} and from the fact that the complex $\zz_{p}(c)=0$
for $c<0$.
\begin{lem}\label{ll}
We have $\Hom(\mathcal{X},\zz/p(c)[d])=0$ if $c<0$. The same formula holds with integral coefficients.
\end{lem}

The following lemma follows from the proof of Sublemma~6.2 of \cite{IzhV00}. Sublemma~6.2 in \cite{IzhV00}
is about quadrics, but the same proof works for any smooth projective irreducible variety.
\begin{lem}[{\cite[Proof of Sublemma~6.2]{IzhV00}}]\label{izhvishik}
Let $X$ be a smooth projective irreducible variety, let $\mathcal{X}$ be the motive of $\check C(X)$, and
$\mu'\in\Hom(\mathcal{X},\mathcal{X}\{b\}[1])$ for some $b$.

Let $\mu\in\Hom(\mathcal{X},\zz_p\{b\}[1])$ be the image of $\mu'$ under the natural map of
Lemma~\ref{localize}(II).

Then for all $c,e$ the pull-back
$$(\mu')^*\colon\Hom(\mathcal{X},\zz_p(c)[e])\to\Hom(\mathcal{X},\zz_p(c+b)[e+2b+1])$$
coincides with the multiplication by $\mu$.
\end{lem}

The next two lemmas follow from the (proven) Bloch--Kato Conjecture.

\begin{lem}[{\cite[Lemma~6.6]{Vo03}}]\label{beilinson2}
Let $X$ be a smooth projective irreducible variety over $k$,
let $\mathcal{X}$ be the motive of $\check C(X)$, and
$\widetilde{\mathcal{X}}=\Cone(\mathcal{X}\to\Spec k)$.

Then $$\widetilde{H}^{d,c}(\widetilde{\mathcal{X}},\zz/p)=0$$
for all $d\le c+1$.
\end{lem}

\begin{lem}[{\cite[Lemma~2.1]{Ro07}}]\label{l1}
Let $X$ be a smooth projective irreducible variety over $k$, $n\ge 2$,
and let $\mathcal{X}$ be the motive of $\check C(X)$. Then there is a natural
exact sequence:
$$0\to H^{n,n-1}(\mathcal{X},\zz/p)\to H^n_{et}(k,\mu_p^{\otimes(n-1)})
\to H^n_{et}(k(X),\mu_p^{\otimes(n-1)}).$$
\end{lem}

In particular, any element of $H^{n,n-1}(\mathcal{X},\zz/p)$
can be identified with an element of the Galois cohomology
$H^n_{et}(k,\mu_p^{\otimes(n-1)})$.

\section{Rost's construction}\label{skeleton}
We assume that $\mathrm{char}\,k=0$ and $k$ contains a primitive $p$-th
root of unity $\zeta_p$.

\begin{dfn}[Norm varieties]
Let $u\in H_{et}^n(k,\mu_p^{\otimes n})$.
A smooth projective irreducible variety $X$ over $k$ is called a norm variety for $u$
if $X$ is a $\nu_{n-1}$-variety and $u_{k(X)}=0$.
\end{dfn}
If $u$ is a pure symbol, i.e., $u=(a_1)\cup\ldots\cup(a_n)$
for some $a_i\in H_{et}^1(k,\mu_p)\simeq k^\times/k^{\times p}$,
then it was shown by M.~Rost that a norm variety for $u$ exists.
There is a general conjecture that if $X$ is a norm variety for $u$,
then $u$ is a pure symbol.

\begin{dfn}[Rost]\label{special}
Let $n$ be a positive integer, $p$ a prime, $b=\frac{p^{n-1}-1}{p-1}$,
and $X$ a smooth projective
geometrically irreducible variety over $k$
of dimension $p^{n-1}-1$. Consider the complex
$$
\CH^b(X)\xrightarrow{\pi_0^*-\pi_1^*} \CH^b(X\times X)\xrightarrow{\pi_0^*-\pi_1^*+\pi_2^*} \CH^b(X\times X\times X)
$$
where $\pi_i$ is the $i$-th projection in the diagram
$$X\llarrow X\times X\lllarrow X\times X\times X.$$
Let $\rho\in\CH^b(X\times X)$. Define
$$c(\rho):=(\pi_0)_*(\rho^{p-1})\in\CH^0(X)=\zz.$$
The cycle $\rho$ is called a {\it special correspondence} of type $(n,p)$ for $X$
if $$(\pi_0^*-\pi_1^*+\pi_2^*)(\rho)=0$$ and $$c(\rho)\ne 0\text{ mod }p.$$
\end{dfn}

In this definition we assume that $X$ is geometrically irreducible, since we consider products
of varieties, like $X\times X$, which can be reducible, if $X$ is irreducible, but not geometrically
irreducible.

M.~Rost showed that a variety $X$ which possesses a special
correspondence and has no zero-cycles of degree coprime to $p$
is a $\nu_{n-1}$-variety (see \cite[Section~9]{Ro07}).
We remark also that by \cite[Lemma~5.2]{Ro07} one has $\rho_{k(X)}=H\times 1-1\times H$
for some cycle $H\in\CH^b(X_{k(X)})$.

The following table gives some examples of norm varieties
(cf. \cite{Ro02}). We refer to \cite{Inv} for notations.
\newline
\begin{longtable}{l|l|l|p{8cm}}
$p$ & $n$ & symbol & norm variety\\
\hline
$2$ & any & $(a_1)\cup\ldots\cup(a_n)$ &
the projective quadric given by $q=0$ where
$q=\lla a_1,\ldots,a_{n-1}\rra\perp\langle -a_n\rangle$.\\
\hline
$2$ & $5$ & $f_5$ & the variety of singular trace zero lines in a reduced Albert algebra
with the cohomological invariant $f_5$; this variety is a twisted form of $\F/P_4$,
where $P_4$ is a maximal parabolic subgroup of type $4$ (enumeration of simple roots
follows Bourbaki).\\
\hline
any & $2$ & $(a)\cup(b)$ &
Severi--Brauer variety $\mathrm{SB}(A)$ where $A=\langle x,y\mid x^p=a,
y^p=b, xy=\zeta_pyx\rangle$.\\
\hline
$3$ & $3$ & $g_3$ & the twisted form of a hyperplane section of
$\mathrm{Gr}(3,6)$ defined in \cite{S08}; this variety is closely related
to Albert algebras with the Rost invariant $g_3$.\\
\hline
$5$ & $3$ & $h_3$ & the variety is related to groups of type $\E_8$
with the Rost invariant $h_3$; it can be constructed from projective
$\E_8$-homogeneous varieties using the algorithm of Section~\ref{seccompr}.\\
\hline
any & $3$ & $(a)\cup(b)\cup(c)$ & any smooth compactification
of the Merkurjev--Suslin variety $\mathrm{MS}(a,b,c)=\{\alpha\in A
\mid\mathrm{Nrd}_A(\alpha)=c\}$ where
$A$ is as above, and $\mathrm{Nrd}_A$ is its reduced norm.\\
\hline
$3$ & $4$ & $(a)\cup(b)\cup(c)\cup(d)$ & any smooth compactification of the variety
$\{\alpha\in J\mid N_J(\alpha)=d\}$ where $J$ is the Albert algebra
obtained from the 1st Tits construction out of $A$ and $c$ with
$A$ as above, and $N_J$ is the cubic norm on $J$.
\end{longtable}

We recall now the construction of the special correspondences of Rost
out of a symbol.
Let $n$ be an integer, $p$ a prime,
$u\in H_{et}^n(k,\mu_p^{\otimes (n-1)})$ an element in the Galois
cohomology of the field $k$, and $X$ a geometrically irreducible smooth projective variety
such that $u_{k(X)}=0$.

By Lemma~\ref{l1} the element $u$ can be identified with an element
$\delta\in H^{n,n-1}(\mathcal{X},\zz/p)$ of the motivic cohomology
of the standard simplicial scheme $\mathcal{X}$ associated with $X$.

Define now $\mu=\tilde Q_0\circ Q_1\circ\ldots\circ Q_{n-2}(\delta)
\in H^{2b+1,b}(\mathcal{X},\zz)$, where $Q_i$ are the Milnor
operations, $\tilde Q_0$ is the integral valued Bockstein,
and $b=\frac{p^{n-1}-1}{p-1}$.

Consider now the {\it skeleton filtration} on $\mathcal{X}$.
Let $Y=\Cone(X\to\mathcal{X})$ be the cone of the natural map.
Then there is another natural map $(X\times X)[1]\to Y$ such that
the composition $(X\times X)[1]\to Y\to X[1]$ equals $\pi_0-\pi_1$.

Let $Z=\Cone((X\times X)[1]\to Y)$. Then again there is a natural map
$(X\times X\times X)[2]\to Z$ such that the composition
$(X\times X\times X)[2]\to Z\to (X\times X)[2]$ equals $\pi_0-\pi_1+\pi_2$.

These exact triangles give the long exact sequences for
motivic cohomology with integral coefficients:
$$H^{2b,b}(X)\to H^{2b+1,b}(Y)\to H^{2b+1,b}(\mathcal{X})\to H^{2b+1,b}(X)$$
and
$$H^{2b+1,b}(Y)\to H^{2b,b}(X\times X)\xrightarrow{\alpha} H^{2b+2,b}(Z).$$

Notice now that $H^{2b,b}(X)=\CH^b(X)$, $H^{2b+1,b}(X)=0$,
$H^{2b,b}(X\times X)=\CH^b(X\times X)$ and there is a map
$$H^{2b+2,b}(Z)\to H^{2b,b}(X\times X\times X)=\CH^b(X\times X\times X)$$
such that the composition $$\CH^b(X\times X)\xrightarrow{\alpha} H^{2b+2,b}(Z)\to
\CH^b(X\times X\times X)$$ equals $\pi_0^*-\pi_1^*+\pi_2^*$.

Summarizing we have the following diagram:
$$
\xymatrix{
\CH^b(X) \ar[rd]_{\pi_0^*-\pi_1^*}\ar[r] & H^{2b+1,b}(Y)\ar[r]\ar[d] & H^{2b+1,b}(\mathcal{X})\ar[r]&0\\
& \CH^b(X\times X)\ar[r]^{\alpha}\ar[rd]_{\pi_0^*-\pi_1^*+\pi_2^*} & H^{2b+2,b}(Z)\ar[d]&\\
&&\CH^b(X\times X\times X)& }
$$

Thus, the element $\mu\in H^{2b+1,b}(\mathcal{X})$ gives rise to
an element $\rho=\rho(\mu)$ in the homology of the complex
$$\CH^b(X)\xrightarrow{\pi_0^*-\pi_1^*}\CH^b(X\times X)
\xrightarrow{\pi_0^*-\pi_1^*+\pi_2^*}\CH^b(X\times X\times X).$$

If the element $\rho$ (as a cycle in $\CH^b(X\times X)$) satisfies
additionally the condition $$c(\rho)\ne 0\text{ mod }p,$$ then $\rho$
turns out to be a special correspondence on $X$.

\section{Norm varieties and special correspondences}\label{mainsection}
We continue to assume that $\mathrm{char}\,k=0$ and $k$ contains a primitive $p$-th
root of unity. The goal of this and of the next sections is to prove the following theorem:
\begin{thm}\label{mainthm}
Let $X$ be a smooth projective geometrically irreducible variety 
which possesses a special correspondence of type $(n,p)$.
Assume that $X$ has no zero-cycles of degree coprime to $p$.
Then there exists a unique up to non-zero scalar (functorial) element $0\ne u\in H_{et}^n(k,\mu_p^{\otimes (n-1)})$
such that $X$ is a norm variety for $u$.

For any field extension $K/k$ the invariant $u_K=0$ iff
$X_K$ has a zero-cycle of degree coprime to $p$.
\end{thm}

As in \cite{Ro07} we define $I(X)$ as the image of the degree map $\deg\colon\CH_0(X)\to\zz$,
where $X$ is a smooth projective variety over $k$.

\begin{prop}[M.~Rost]\label{rostmotive}
Let $X$ be a smooth projective geometrically irreducible variety which possesses a special correspondence $\rho$
of type $(n,p)$.
Assume $I(X)\subset p\zz$ and set $b=\tfrac{p^{n-1}-1}{p-1}$. Then the motive of $X$ considered with
$\zz_{(p)}$-coefficients has an indecomposable generically split (see Definition~\ref{gensplmot} below) direct
summand $R$ constructed in \cite{Ro07} such that
$$R\otimes X\simeq\bigoplus_{i=0}^{p-1}X\{bi\}.$$
\end{prop}

This Proposition is proved in \cite[Proposition~7.14]{Ro07}.
The motive $R$ is called a {\it generalized Rost motive}. We outline
Rost's proof for reader's convenience to make the exposition more self-contained.

We write $\CH$ for the Chow group with $\zz_{(p)}$-coefficients.
By Manin's identity principle (Yoneda lemma for motives) it suffices to
construct a correspondence
$$\theta\in
\Hom(\bigoplus_{i=0}^{p-1}X\{bi\},X\otimes R)$$ such that
for any smooth projective variety $T$ over $k$
$$\theta\circ-\colon\Hom(T,\bigoplus_{i=0}^{p-1} X\{bi\})\to
\Hom(T,X\otimes R)$$ is an isomorphism. Notice that the composition with
$\theta$ equals the realization of $\Delta_T\otimes\theta$.

Define $$\theta_i=(\Delta_X\times 1)(1\times\rho^i)\in\Hom(X\{b(p-1-i)\},X\otimes R),\quad i=0,\ldots,p-1,$$
and $$\theta=\sum_{i=0}^{p-1}\theta_i.$$

Let $Y=T\times X$ and $f\colon Y\to X$ be the projection. For $0\le i\le p-1$
let
\begin{align*}
\varphi_i\colon\CH^h(Y)&\to\CH^{h+bi}(Y\times X)\\
\alpha&\mapsto\pi_X^*(\alpha)\cdot(f\times\id_X)^*(\rho^i),
\end{align*}
where $\pi_X\colon Y\times X\to Y$ is the projection,
$$\Phi=\sum_{i=0}^{p-1}\varphi_i$$ and
\begin{align*}
\psi_i\colon\CH^{h+bi}(Y\times X)&\to\CH^h(Y)\\
\beta&\mapsto(\pi_X)_*(\beta\cdot(f\times\id_X)^*(\rho^{p-1-i})),
\end{align*}
$$\Psi=\sum_{i=0}^{p-1}\psi_i.$$

One verifies using projection formulas that the realization of $\Delta_T\otimes \theta$
equals $\Phi$. In \cite[Proof of Proposition~7.14]{Ro07} Rost shows that
$\Phi$ is an isomorphism. Therefore $\theta$ is an isomorphism.

Observe also the following: Let $\BX=X_{k(X)}$ and $\overline R=R_{k(X)}$. Then a direct computation shows that over $k(X)$ (where $R$ becomes
$\bigoplus_{i=0}^{p-1}\zz_{(p)}\{bi\}$) the constructed isomorphism
$$\bigoplus_{i=0}^{p-1}\overline X\{bi\}\xrightarrow{\theta_{k(X)}}
\overline X\otimes \overline R=\bigoplus_{i=0}^{p-1}\overline X\{bi\}$$ is a lower triangular matrix
with $\Delta_{\overline X}$'s up to invertible scalars on the diagonal, i.e., $\theta_{k(X)}\colon \BX\{bi\}\to\BX\{bj\}$ is zero
for $i>j$, and equals $\Delta_{\overline X}$ up to an invertible scalar for $i=j$.
Moreover, Rost shows that $\overline R\simeq (\overline X,c\cdot \rho^{p-1})\mod p$ for some scalar $c$.

The following lemma will {\it not} be used in the construction of a cohomological invariant for groups of type $\E_8$
given in Section~\ref{sec8} below. This lemma was inspired by \cite[Proof of Thm.~4.4]{Vo01} and
\cite[Proof of Statement~1.1.1]{Vi98}.

\begin{lem}\label{triang}
Let $X$ be a smooth projective geometrically irreducible
variety which possesses a special correspondence $\rho$ of type $(n,p)$. Assume $I(X)\subset p\zz$ and let $R$ be a direct summand
of the motive of $X$ with $\zz_{p}$-coefficients given in Proposition~\ref{rostmotive}.
Then there exists the following filtration of $R$ consisting
of exact triangles:
\begin{gather*}
R_{p-2}\{b\}\to R_{p-1}\to\mathcal{X}\\
R_{p-3}\{b\}\to R_{p-2}\to\mathcal{X}\\
\ldots\\
R_0\{b\}\to R_1\to\mathcal{X}\\
0\to R_0\to\mathcal{X},
\end{gather*}
where $R_{p-1}=R$ and $\mathcal{X}$ denotes the motive of $\check C(X)$.
\end{lem}
\begin{proof}
In this proof we write $\CH$ for Chow groups with $\zz_p$-coefficients.
The proof goes by induction.

First of all, notice that $\mathcal{X}$ lies in the localizing subcategory generated by $X$.
Let $M$ be a motive in the localizing subcategory generated by $X$
such that there is an isomorphism $$\theta\colon\bigoplus_{i\in I} X\{i\}\xrightarrow{\simeq}M\otimes X$$
for some finite set $I$ of non-negative integers.
Denote $\overline X=X\times_kk(X)$ and
take an isomorphism $\upsilon=(\upsilon_i)_{i\in I}\colon M_{k(X)}\to\bigoplus_{i\in I}\zz_{p}\{i\}$
defined over $k(X)$. Notice that the isomorphism $\upsilon=(\upsilon_i)$ is unique up to
a sequence of invertible scalars, since
$\Hom(\zz_p\{i\},\zz_p\{j\})=\zz_p$ if $i=j$ and is $0$ otherwise.

Assume additionally that the composite morphism
$$
\bigoplus_{i\in I}\BX\{i\}\xrightarrow{\theta_{k(X)}}(M\otimes X)_{k(X)}
\xrightarrow{\upsilon\otimes\id_{\BX}}\bigoplus_{i\in I}\BX\{i\}
$$
is a lower triangular matrix with isomorphisms
on the diagonal.

Observe that these conditions are satisfied by
the motive $R=R_{p-1}$ given in the previous proposition.
So, we can start induction with $M=R_{p-1}$ and $I=\{bi, i=0,\ldots,p-1\}$.

Consider the skeleton filtration of $\mathcal{X}$.
Let $Y=\Cone(X\to\mathcal{X})$ be the cone of the canonical
map $X\to\mathcal{X}$ induced by the structural map $X\to\Spec k$.
Then we have the following exact triangle:
\begin{equation}\label{tensormotive}M\otimes X\to M\otimes\mathcal{X}\to M\otimes Y.
\end{equation}

By Lemma~\ref{localize}(I) $M\otimes\mathcal{X}=M$.
Let now $l$ be the smallest element in $I$.

The above triangle induces the long exact sequence
\begin{multline}\label{f3}
\Hom(M\otimes Y,\zz_{p}\{l\})\to\Hom(M,\zz_{p}\{l\})\to\Hom(M\otimes X,\zz_{p}\{l\})\\
\to\Hom((M\otimes Y)[-1],\zz_{p}\{l\}).
\end{multline}

The group $\Hom(M\otimes Y,\zz_{p}\{l\})=0$ by \cite[Prop.~8.1]{Vo01}, since $$\Hom(V\{l'\}[j],\zz_{p}\{l\})=\Hom(V[j],\zz_{p}\{l-l'\})=0$$
for any smooth projective variety $V$ and any $j>0$ and $l'\ge l$ (see Lemma~\ref{zerohom})
and since $Y$ is an extension of $X^i[i-1]$, $i>1$.

Moreover, the same arguments applied to the exact triangle $(X\times X)[1]\to Y\to Z$, where $Z$ denotes
the respective cone, show that we have an injection
$$\Hom((M\otimes Y)[-1],\zz_{p}\{l\})\to \Hom(M\otimes X\otimes X,\zz_{p}\{l\})$$
induced by the natural map $(X\otimes X)[1]\to Y$. Observe that the composition
$$(X\otimes X)[1]\to Y\to X[1]$$ is the difference of two projections
$\pi_0-\pi_1$ (see Section~\ref{skeleton} and \cite[p.~31]{Vi98}).

Since by induction hypothesis $M\otimes X\simeq \bigoplus_{i\in I}X\{i\}$
and since $l$ is minimal, we
have $\Hom(M\otimes X,\zz_{p}\{l\})\simeq\Hom(X\{l\},\zz_{p}\{l\})=\Hom(X,\zz_{p})=\CH^0(X)$.

We claim next that the image of the structural map $\pr\colon X\to\Spec k$
(as an element of $\Hom(M\otimes X,\zz_{p}\{l\})$) in
$\Hom((M\otimes Y)[-1],\zz_{p}\{l\})$ is trivial. Indeed, 
by the above identifications the morphism $\Hom(M\otimes X,\zz_p\{l\})\to\Hom(M\otimes X\otimes X,\zz_p\{l\})$
corresponds to the morphism $\CH^0(X)\to\CH^0(X\times X)$ given by $\pi_0^*-\pi_1^*$, which sends $\pr=[X]$ to $0$.
Therefore, since sequence~\eqref{f3} is exact, we can find a preimage
$\beta\in\Hom(M,\zz_{p}\{l\})$ of the structural map.

Lemma~\ref{localize}(II) gives us a map $M\to\mathcal{X}\{l\}$.
Define $$M'=\Cone(M\to\mathcal{X}\{l\})[-1]$$
and consider the composite map
$$\gamma\colon M'\otimes X\xrightarrow{\alpha\otimes\id_X} M\otimes X\xrightarrow[\theta^{-1}]{\simeq}\bigoplus_{i\in I}X\{i\}\to\bigoplus_{i\in I\setminus\{l\}}X\{i\}$$
arising from the exact triangle $(M'\xrightarrow{\alpha} M\to\mathcal{X}\{l\})\otimes X$,
where the last map in $\gamma$ is the canonical projection.
By our inductive assumptions on $M$ one can see that
over $k(X)$ where the motives $M$ and $M'$ become sums of twisted Tate motives,
$\gamma_{k(X)}$ is an isomorphism given by a lower triangular matrix with isomorphisms on the diagonal.
Therefore by \cite[Theorem~2.3.5]{Vi98}
we have an isomorphism $M'\otimes X\otimes X\simeq\bigoplus_{i\in I\setminus\{l\}}X\otimes X\{i\}$.
So, we can replace $X$ by $X\times X$ and apply induction. Notice that the motives of $\check C(X)$
and of $\check C(X\times X)$ are isomorphic by Lemma~\ref{lemma33}, and $\gamma$ satisfies all
the hypothesis needed for the next inductive step.

Therefore by induction we can construct the motives
$$M_{p-1}:=R,M_{p-2},M_{p-3},\ldots,M_0,M_{-1}$$ together with exact triangles
$$M_{s-1}\to M_s\to\mathcal{X}\{b(p-1-s)\},\quad s=0,\ldots,p-1,$$ and such
that $M_s\otimes X^{p-s}\simeq\bigoplus_{i=p-1-s}^{p-1}X^{p-s}\{bi\}$ for all $s$.
In particular, $M_{-1}\otimes X^{p+1}=0$.

But then $M_{-1}=M_{-1}\otimes\mathcal{X}=0$. Finally,
the existence of exact triangles as in the statement of the lemma follows from the
cancellation theorem (see Lemma~\ref{cancellation}) with $R_{s}\{b(p-1-s)\}\simeq M_{s}$ for all $s$.
We are done.
\end{proof}

Next we investigate the construction of Lemma~\ref{triang} in more details.

\begin{lem}\label{triv}
For $i=1,\ldots,p-1$
if $c_1>c_2$ then $$\Hom(R_i(c_1)[d_1],\zz/p(c_2)[d_2])=0.$$
The same formula holds for integral coefficients.
\end{lem}
\begin{proof}
The statement follows from Lemma~\ref{triang} and Lemma~\ref{ll}.
\end{proof}

\begin{lem}\label{defeta}
Let $X$ be a smooth projective irreducible variety over $k$ and $R$ a direct summand of the motive of $X$ with $\zz_p$-coefficients.
Assume that $I(X)\subset p\zz$.
Denote by $\mathcal{X}$ the motive of the standard simplicial scheme associated with $X$.
Assume that one of the following two conditions holds:

I) The motive $R$ arises from a special correspondence $\rho$ of type $(n,p)$ (in particular, in this case
we assume that $X$ is geometrically irreducible) or

II) $p=2$ and the motive $R$ has a filtration $\mathcal{X}\{b\}\to R\to\mathcal{X}$
for some $b>0$.

Let $K/k$ be a field extension. Then the group
$\Hom(\mathcal{X}_K,\zz_p\{b\}[1])$ is finite cyclic of
order $1$ or $p$, and
$\Hom(\mathcal{X}_K,\zz_p\{b\}[1])=0$ iff $X_K$ has a zero-cycle
of degree coprime to $p$.
\end{lem}
\begin{proof}
We write $\CH$ for Chow groups with $\zz_p$-coefficients.
By functoriality of all our constructions it suffices to prove the lemma for $K=k$.

In case I of the present lemma we use notation from Lemma~\ref{triang}, in particular, $R=R_{p-1}$.
In case II we set $R_{p-1}:=R$, $R_{p-2}:=\mathcal{X}$, and $R_{p-3}:=0$.

Consider $\Hom$ from the exact triangle $R_{p-2}\{b\}\to R_{p-1}\to\mathcal{X}$
to $\zz_p\{b\}$. We get the following long exact sequence:
\begin{align*}
\Hom(R_{p-1},\zz_p\{b\})\xrightarrow{f}\Hom(R_{p-2}\{b\},\zz_p\{b\})\\
\xrightarrow{g}\Hom(\mathcal{X}[-1],\zz_p\{b\})\to\Hom(R_{p-1}[-1],\zz_p\{b\}).
\end{align*}
Notice that $\Hom(R_{p-1},\zz_p\{b\})=\CH^b(R_{p-1})$ (since $R_{p-1}$ is a pure motive),\\
$\Hom(R_{p-2}\{b\},\zz_p\{b\})=\Hom(R_{p-2},\zz_p)$ (Lemma~\ref{cancellation}),\\
$\Hom(\mathcal{X}[-1],\zz_p\{b\})=\Hom(\mathcal{X},\zz_p\{b\}[1])$, and\\
$\Hom(R_{p-1}[-1],\zz_p\{b\})\le\Hom(X[-1],\zz_p\{b\})=0$ (Lemma~\ref{zerohom}).

Assume that we are in Case I.
It follows from the proof of Lemma~\ref{triang} that
the homomorphism $$\Hom(R_{p-2},\zz_p)\to\Hom(R_{p-2}\otimes X,\zz_p)=\CH^0(X)$$ coming from
exact sequence~\eqref{f3} (with $M=R_{p-2}\{b\}$) is an isomorphism.
Thus, we get the following commutative diagram with exact rows

$$\xymatrix{
\Hom(R_{p-1},\zz_p\{b\}) \ar[r]^-f\ar@{=}[d] & \Hom(R_{p-2}\{b\},\zz_p\{b\})
\ar[r]^-g\ar[d]^{=} &\Hom(\mathcal{X},\zz_p\{b\}[1]) \ar[r]\ar@{=}[d]& 0\\
\CH^b(R_{p-1}) \ar[r]^-f& \CH^b(X\{b\}) \ar[r]^-g&\Hom(\mathcal{X},\zz_p\{b\}[1])\ar[r] & 0
}$$

We take now the extension of scalars from $k$ to $k(X)$. We have the following commutative diagram
with exact rows

$$\xymatrix{
\CH^b(R_{p-1}) \ar[r]^-f\ar[d]& \CH^0(X) \ar[r]^-g\ar[d]^{=}&\Hom(\mathcal{X},\zz_p\{b\}[1])\ar[r]\ar[d] & 0\\
\CH^b((R_{p-1})_{k(X)}) \ar[r]^-{f_{k(X)}}& \CH^0(X_{k(X)})\ar[r]&0&
}$$

The morphism $f_{k(X)}$ is surjective, since by Lemma~\ref{lemma33} the motives $\mathcal{X}_{k(X)}$ and $\Spec k(X)$
are isomorphic and hence by Lemma~\ref{zerohom} $\Hom(\mathcal{X}_{k(X)},\zz_p\{b\}[1])=0$.

On the other hand, by \cite{Ro07} (see also \cite[Lemma~SC.3]{KM13})
the image of the restriction homomorphism $$\CH^b(R_{p-1})\to\CH^b((R_{p-1})_{k(X)})=\zz_p$$ (the last equality
is due to the fact that $R_{p-1}$ is isomorphic over $k(X)$ to $\oplus_{i=0}^{p-1}\zz_p\{bi\}$)
contains $p\zz_p$. In particular, the group $\Hom(\mathcal{X},\zz_p\{b\}[1])$ has order $1$ or $p$ (and, in particular,
it is cyclic).

If $X$ has a zero-cycle of degree coprime to $p$, then by Lemma~\ref{lemma33}
$$\Hom(\mathcal{X},\zz_p\{b\}[1])=0.$$ Conversely, if $\Hom(\mathcal{X},\zz_p\{b\}[1])=0$,
then the restriction homomorphism $$\CH^b(R_{p-1})\to\CH^b((R_{p-1})_{k(X)})$$ is surjective and hence $X_{k(X)}$
has a zero-cycle of degree coprime to $p$ defined over $k$, namely, the cycle $H^{p-1}$, where the special correspondence
$\rho_{k(X)}=H\times 1-1\times H$ (see \cite[Lemma~5.2]{Ro07}).

Assume now we are in Case II. Proceeding as in Case I we have a commutative diagram with exact rows:

$$\xymatrix{
\CH^b(R) \ar[r]\ar[d]& \Hom(\mathcal{X}\{b\}[1],\zz_2\{b\}[1]) \ar[r]\ar[d]^{=}&\Hom(\mathcal{X},\zz_2\{b\}[1])\ar[r]\ar[d] & 0\\
\CH^b(R_{F}) \ar[r]& \zz_2\ar[r]&0&
}$$
\noindent
where $F/k$ is a field extension over which $X$ has a zero-cycle of odd degree.

If $X$ has a zero-cycle of odd degree, then exactly as in Case I the group $$\Hom(\mathcal{X},\zz_2\{b\}[1])=0.$$
Conversely, if this group is zero, then the restriction homomorphism $\CH^b(R)\to\CH^b(R_F)$
is surjective. 

Let $R=(X,\pi)$. Since over $F$ the motive $R_F$ is isomorphic to $\zz_2\oplus\zz_2\{b\}$,
the projector $\pi_F$ equals to $1\times\pt+x\times y$, where $\pt$ is a zero-cycle on $X_F$ of degree $1$,
$x\in\CH^b(R_F)$ and $y\in\CH_b(R_F)$. The cycle $y$ is rational by \cite[Lemma~3.5]{GPS}
(applied to the transposed projector $\pi^t$) and the cycle $x$ is rational by the
above considerations. Therefore $X_F$ has a zero-cycle $x\cdot y$ of degree $1$, which is defined over $k$.

Finally, by the proof of the next Lemma~\ref{transfer} (which does not use the present Lemma), the
group $\Hom(\mathcal{X},\zz_2\{b\}[1])$ is a $2$-torsion group. Therefore, if this group is non-trivial, then
it has order $2$. This implies the lemma.
\end{proof}

\begin{lem}\label{transfer}
In the settings of Lemma~\ref{defeta} the natural homomorphism
$$\Hom(\mathcal{X},\zz_p(f)[g])\to\Hom(\mathcal{X},\zz/p(f)[g])$$ is injective for $g>f$.
\end{lem}
\begin{proof}
Let $R=(X,\pi)$. Since $R$ is a direct summand of $X$, the $0$-cycle $(\pi_0)_*(\pi\cdot\pi^t)$,
where $\pi_0\colon X^2\to X$ is the first projection, has degree $p$ (since $R$ is a sum of $p$ twisted Tate motives
over an extension of the base field).
In particular, there exists a finite field extension $K/k$ of degree $pm$
with $m$ coprime to $p$ such that $X_K$ has a zero-cycle of degree coprime to $p$.
Therefore by Lemma~\ref{lemma33} $\mathcal{X}_K\simeq\Spec K$ and hence $\Hom(\mathcal{X}_K,\zz_p(f)[g])=0$
by Lemma~\ref{zerohom}.

By a transfer argument the composite map
$$\Hom(\mathcal{X},\zz_p(f)[g])\xrightarrow{\res}\Hom(\mathcal{X}_K,\zz_p(f)[g])\xrightarrow{\mathrm{tr}}\Hom(\mathcal{X},\zz_p(f)[g])$$
where the first map is the extension of scalars and the second map is the transfer map,
is the multiplication by $pm$. Therefore $\Hom(\mathcal{X},\zz_p(f)[g])$ is a $p$-torsion group.

The exact sequence
$$\Hom(\mathcal{X},\zz_p(f)[g])\xrightarrow{\cdot p}\Hom(\mathcal{X},\zz_p(f)[g])\to\Hom(\mathcal{X},\zz/p(f)[g])$$
implies now the statement of the lemma.
\end{proof}

\begin{lem}\label{le57}
In the settings of Lemma~\ref{defeta}
the natural homomorphism $$\Hom(\mathcal{X},\zz_p\{b\}[1])\to\Hom(\mathcal{X},\zz/p\{b\}[1])$$ is an isomorphism.
\end{lem}
\begin{proof}
The injectivity of this homomorphism follows from Lemma~\ref{transfer}.
To prove surjectivity it suffices to show that $\Hom(\mathcal{X},\zz_p\{b\}[2])=0$.

We use notation from the proof of Lemma~\ref{defeta}.
Consider $\Hom$ from the exact triangle $R_{p-2}\{b\}\to R_{p-1}\to\mathcal{X}$ to $\zz_p\{b\}[2]$.
We get the following long exact sequence:
$$\Hom(R_{p-2}(b)[2b+1],\zz_p(b)[2b+2])\to\Hom(\mathcal{X},\zz_p(b)[2b+2])\to\Hom(R_{p-1},\zz_p(b)[2b+2]).$$

The last group of this sequence is zero by Lemma~\ref{zerohom}, since $R_{p-1}$ is a pure motive.
The first group of this sequence equals $\Hom(R_{p-2},\zz_p[1])$ by cancellation.

Consider now $\Hom$ from the exact triangle $R_{p-3}\{b\}\to R_{p-2}\to\mathcal{X}$ to $\zz_p[1]$.
We get the following exact sequence:
$$\Hom(\mathcal{X},\zz_p[1])\to\Hom(R_{p-2},\zz_p[1])\to\Hom(R_{p-3}(b)[2b],\zz_p[1]).$$

The last group of this sequence is zero by Lemma~\ref{triv}.
By \cite[Cor.~6.9]{Vo01} and \cite[Prop.~2.7]{Vo96} the first group of this sequence is isomorphic to
$H_{et}^{1,0}(\mathcal{X},\zz_p)\simeq H_{et}^{1,0}(k,\zz_p)=0$.
Therefore $\Hom(R_{p-2},\zz_p[1])=0$ and hence $$\Hom(\mathcal{X},\zz_p\{b\}[2])=0.$$
\end{proof}

Let $\widetilde{\mathcal{X}}=\Cone(\mathcal{X}\to\Spec k)$.

\begin{lem}\label{milnortriv}
In the settings of lemma~\ref{defeta} assume additionally in case II that $b>2^{n-2}-1$.
Define $$\mu\in\Hom(\mathcal{X},\zz/p\{b\}[1])=\widetilde{H}^{2b+2,b}(\widetilde{\mathcal{X}},\zz/p)$$
as the image of a generator of the cyclic group $\Hom(\mathcal{X},\zz_p\{b\}[1])$ under
the natural map $\Hom(\mathcal{X},\zz_p\{b\}[1])\xrightarrow{\simeq}\Hom(\mathcal{X},\zz/p\{b\}[1])$.
Then for all $0\le i\le n-2$ we have $Q_i(\mu)=0$.
\end{lem}
\begin{proof}
We use notation from the proof of Lemma~\ref{defeta}.

The exact triangle $R_{p-1}\to\mathcal{X}\to R_{p-2}\{b\}[1]$ gives
the exact sequence
\begin{align*}
&\Hom(R_{p-2}\{b\}[1],\zz/p(b+p^i-1)[2b+2p^i])\to\\
&\to\Hom(\mathcal{X},\zz/p(b+p^i-1)[2b+2p^i])\to\Hom(R_{p-1},\zz/p(b+p^i-1)[2b+2p^i]),
\end{align*}
and $Q_i(\mu)\in\Hom(\mathcal{X},\zz/p(b+p^i-1)[2b+2p^i])$.

Since $2(b+p^i-1)<2b+2p^i$ and $R_{p-1}$ is a pure motive, the group
$$\Hom(R_{p-1},\zz/p(b+p^i-1)[2b+2p^i])=0$$ by Lemma~\ref{zerohom}.

Therefore in order to prove that $Q_i(\mu)=0$, it suffices to show that the group
$$\Hom(R_{p-2}\{b\}[1],\zz/p(b+p^i-1)[2b+2p^i])=\Hom(R_{p-2},\zz/p(p^i-1)[2p^i-1])=0.$$

The same exact triangle $R_{p-1}\to\mathcal{X}\to R_{p-2}\{b\}[1]$ gives
the exact sequence
\begin{align*}
\Hom(R_{p-2}(b)[2b+1],\zz/p(p^i-1)[2p^i-1])\to\Hom(\mathcal{X},\zz/p(p^i-1)[2p^i-1])\\
\to\Hom(R_{p-1},\zz/p(p^i-1)[2p^i-1]).
\end{align*}
The first group in this sequence is zero by Lemma~\ref{triv}, since
$b>p^i-1$ by our assumptions. The last group is zero, since $R_{p-1}$ is a pure
motive and $2p^i-1>2(p^i-1)$. Therefore $\Hom(\mathcal{X},\zz/p(p^i-1)[2p^i-1])=0$.

The exact triangle $R_{p-3}\{b\}\to R_{p-2}\to\mathcal{X}$ gives the
exact sequence
\begin{align*}
\Hom(\mathcal{X},\zz/p(p^i-1)[2p^i-1])\to\Hom(R_{p-2},\zz/p(p^i-1)[2p^i-1])\\
\to\Hom(R_{p-3}\{b\},\zz/p(p^i-1)[2p^i-1]).
\end{align*}
The first and the last groups are zero by Lemma~\ref{triv} and by the previous
considerations. The lemma is proved.
\end{proof}

The following lemma is due to V.~Voevodsky. We reproduce its proof
for reader's convenience to make the exposition more self-contained.

\begin{lem}[{\cite[Proof of Lemma~6.7]{Vo03}}]\label{llll}
In the settings of Lemma~\ref{milnortriv} assume additionally in case II
that there exists a $\nu_{n-1}$-variety $Y$ and a morphism $Y\to X$.
Let $\delta\in H^{n,n-1}(\mathcal{X},\zz/p)$.
If $Q_{n-2}\ldots Q_1Q_0(\delta)=0$, then $\delta=0$.
\end{lem}
\begin{proof}
The exact triangle of pointed simplicial schemes
$$\mathcal{X}_+\to\Spec k_+\to\widetilde{\mathcal{X}}$$
defines an isomorphism $H^{d,c}(\mathcal{X},\zz/p)\to
\widetilde{H}^{d+1,c}(\widetilde{\mathcal{X}},\zz/p)$ for any $d>c$.
Thus, we can identify
$\delta$ with its image $\widetilde{\delta}$ in
$\widetilde{H}^{n+1,n-1}(\widetilde{\mathcal{X}},\zz/p)$.

Assume that $\widetilde{\delta}\ne 0$.
We want to show that $Q_iQ_{i-1}\ldots Q_0(\widetilde\delta)\ne 0$
for all $i\le n-2$.
Assume by induction that $Q_{i-1}\ldots Q_0(\widetilde{\delta})\ne 0$.
If $Q_iQ_{i-1}\ldots Q_0(\widetilde\delta)=0$, then by Lemma~\ref{margolislemma}
and by Lemma~\ref{milnortriv}
there exists $v$ such that
$Q_i(v)=Q_{i-1}\ldots Q_0(\widetilde\delta)$. 

A straightforward computation shows that
$v\in\widetilde{H}^{d,c}(\widetilde{\mathcal{X}},\zz/p)$
where $$c=n-i-p^i+1+\tfrac{p^i-p}{p-1}\text{ and } d=n-i+2-2p^i+2\cdot\tfrac{p^i-p}{p-1}$$

But $d-c\le 1$. Therefore by Lemma~\ref{beilinson2} $v=0$. Contradiction.
\end{proof}

\begin{lem}\label{lll}
Assume that $p>2$. 
Then in the settings of Lemma~\ref{llll}
there exists an element $\delta\in\Hom(\mathcal{X},\zz/p(n-1)[n])$
such that $\mu=Q_{n-2}\ldots Q_1Q_0(\delta)$.
\end{lem}
\begin{proof}
As in the previous lemma it suffices to prove this identity
in the reduced motivic cohomology of $\widetilde{\mathcal{X}}$.
Let $\eta\in\widetilde H^{2b+2,b}(\widetilde{\mathcal{X}},\zz/p)$
be the image of $\mu$.

Assume we have constructed a sequence of elements
$\eta_{n-2},\eta_{n-3},\ldots,\eta_{n-s}$ for some $2\le s<n$
such that $\eta_{n-i+1}=Q_{n-i}(\eta_{n-i})$ and $$\eta=Q_{n-2}Q_{n-3}\ldots
Q_{n-s}(\eta_{n-s}).$$

For $s=2$ such an element $\eta_{n-2}$ with $\eta=Q_{n-2}(\eta_{n-2})$ exists
by Lemma~\ref{margolislemma}, since $Q_{n-2}(\eta)=0$ by Lemma~\ref{milnortriv}.

We claim that there exists $\eta_{n-s-1}$ satisfying the same
conditions. To prove this it suffices
to show that $v:=Q_{n-s-1}(\eta_{n-s})$ equals $0$.

Consider
\begin{align*}
Q_{n-2}Q_{n-3}\ldots Q_{n-s}(v)=Q_{n-2}Q_{n-3}\ldots Q_{n-s}Q_{n-s-1}(\eta_{n-s})\\
=\pm Q_{n-s-1}(Q_{n-2}Q_{n-3}\ldots Q_{n-s}(\eta_{n-s}))
=\pm Q_{n-s-1}(\eta)=0,
\end{align*}
since $0\le n-s-1\le n-3$.

A straightforward computation shows that
$$\eta_{n-s}\in\widetilde H^{2b+s-2\cdot\tfrac{p^{n-1}-p^{n-s}}{p-1}+1,b+s-1-
\tfrac{p^{n-1}-p^{n-s}}{p-1}}(\widetilde{\mathcal{X}},\zz/p)$$
and
$$v\in\widetilde H^{2b+s-2\cdot\tfrac{p^{n-1}-p^{n-s}}{p-1}+2p^{n-s-1},
b+s-2-\tfrac{p^{n-1}-p^{n-s}}{p-1}+p^{n-s-1}}(\widetilde{\mathcal{X}},\zz/p).$$

Let $2\le t\le s$. We claim that if $Q_{n-t}(Q_{n-t-1}\ldots Q_{n-s}(v))=0$,
then $$Q_{n-t-1}\ldots Q_{n-s}(v)=0.$$ Indeed, if $Q_{n-t}(Q_{n-t-1}\ldots Q_{n-s}(v))=0$,
then there exists $w$ such that $$Q_{n-t}(w)=Q_{n-t-1}\ldots Q_{n-s}(v).$$
A straightforward computation shows that
$w\in\widetilde H^{d,c}(\widetilde{\mathcal{X}},\zz/p)$,
where $$c=\tfrac{p^{n-t}-1}{p-1}-1+p^{n-s-1}+t-p^{n-t}$$ and
$$d=2\cdot\tfrac{p^{n-t}-1}{p-1}+2p^{n-s-1}+t-2p^{n-t}+1.$$
Another straightforward computation shows that $d\le c+1$ if $p>2$,
and by Lemma~\ref{beilinson2} $w=0$. Therefore $v=0$.
\end{proof}

\noindent
{\it Proof of Theorem~\ref{mainthm} for $p>2$:}
The element $\delta$ constructed in Lemma~\ref{lll} can be identified with
an element $u\in H_{et}^n(k,\mu_p^{\otimes(n-1)})$ by Lemma~\ref{l1}.
It remains to show that for any field extension $K/k$ 
$$u_K=0\text{ iff }X_K\text{ has a zero-cycle of degree coprime to }p.$$
By construction $u_K=0$ iff $\delta_K=0$. By Lemma~\ref{llll}
$\delta_K=0$ iff $\mu_K=0$. By Lemma~\ref{defeta}
$\mu_K=0$ iff $X_K$ has a zero-cycle of degree coprime to $p$.

Finally, by Lemma~\ref{defeta} and by Lemma~\ref{le57} the group $H^{2b+1,b}(\mathcal{X},\zz/p)$ and hence
the group $H^{n,n-1}(\mathcal{X},\zz/p)$ is cyclic. This implies
the uniqueness of $u$. The theorem is proved.

\section{Binary motives}\label{sec6}
In this section we investigate the structure of binary direct summands of
smooth projective varieties, i.e., of motives which over a field extension of the base field
become isomorphic to a direct sum of {\it two} twisted Tate motives. In this section we assume
that the characteristic of the base field $k$ is $0$.

\begin{thm}\label{binmot}
Let $X$ be a smooth projective irreducible variety over a field $k$ of characteristic $0$ with no
zero-cycles of odd degree. Denote by $\mathcal{X}$ the motive of the standard simplicial scheme associated with $X$.
Let $M$ be a direct summand of $X$ such that
we have an exact triangle
\begin{align}\label{eqtr61}
\mathcal{X}\{d\}\to M\to\mathcal{X}\xrightarrow{\mu'}\mathcal{X}\{d\}[1]
\end{align}
in the category $\DM$ with $\zz_2$-coefficients.
Assume that one of the following two conditions holds:

(a) $\dim X=d>0$ or

(b) there exists a $\nu_{n-1}$-variety $Y$ for some $n\ge 2$ and a morphism $Y\to X$, the motive $M\{\dim X-d\}$ is a
direct summand of $X$ in the category of Chow motives with $\zz_2$-coefficients, and $2^{n-2}-1<d<2^{n-1}$.

Then
\begin{enumerate}
\item[(I)] $d=2^{n-1}-1$ for some $n\ge 2$ in case (a) and $d=2^{n-1}-1$ in case (b).
\item[(II)] There exists a (functorial) element $u\in H_{et}^n(k,\zz/2)$
such that for any field extension $K/k$
we have $u_K=0$ iff $X_K$ has a zero-cycle of odd degree.
\end{enumerate}
\end{thm}

One can notice that Rost's proof
(\cite[Proof of Theorem~9.1 and Lemma~9.10]{Ro07})
implies the following lemma. For reader's convenience we sketch
Rost's proof below to make the exposition more self-contained.
\begin{lem}\label{rostbinary}
Let $X$ be a smooth projective irreducible variety over $k$ of dimension $d>0$
with no zero-cycles of odd degree which possesses a
cycle $r\in\CH_d(X\times X)$ such that over any field extension $F/k$ over which $X$ has a rational point,
$r_{F}$ mod $2$ is a projector
and $(X_{F},r_{F})\simeq\zz/2\oplus\zz/2\{d\}$ in the category of Chow motives
with $\zz/2$-coefficients.

Then $d=2^{n-1}-1$ for some $n$ and
$X$ is a $\nu_{n-1}$-variety.
\end{lem}
\begin{proof}
Let $F$ be the function field of $X$.
By assumptions, the cycle $r_{F}\mod 2$ equals $$H\times 1+1\times H',$$ where $H$ and $H'$
are zero-cycles on $X_F$ of degree $1$ mod $2$. Substituting $r$ by $r^t\circ r$
we can assume without loss of generality that $H'=H$.

Set $p=2$ and let $S_\bullet$ (resp. $S^\bullet$) denote the
homological (resp. cohomological) Steenrod operations
in the Chow theory modulo $p$ (see \cite{Br03}).

Then $S_\bullet(\alpha)=S^\bullet(\alpha)b_\bullet(X)$, $\alpha\in\Ch(X)$,
where $b_\bullet$ is the (total) characteristic class defined in section~\ref{charnum}
for partition $(p-1,\ldots,p-1)$
(Since $p=2$, we have $b_\bullet(X)=c_\bullet(-T_X)$).

One has $b_d(X)=S_d([X])$. By our assumptions the cycle
$r$ has the property $$(\pi_0)_*(r)=c(r)[X]$$ with $c(r)=1$ mod $2$, where $\pi_0\colon X^2\to X$
is the first projection.
Moreover, we have $S_\bullet(r)=S^\bullet(r)b_\bullet(X^2)$.

We show first that $\deg b_d(X)\ne 0$ mod $p^2$.
To do this,
it suffices to find an integral representation of
the cycle $S_d(r)$ with degree not divisible by $4$ (see \cite[Proof of Thm.~9.1]{Ro07} for more details).

We have
\begin{equation}\label{rostform}
S_d(r)=\sum_{i=0}^d S^i(r)b_{d-i}(X^2).
\end{equation}
The last term
$S^d(r)=r^2$ mod $2$ and $\deg r^2=2c(r)^2$ is not divisible by $4$. Therefore
it suffices to show that all other terms of this sum
have integral representatives whose degrees are divisible by $4$.

For the first term of~\eqref{rostform} we have
$$r_Fb_d(X^2)=H\times b_d(X)+b_d(X)\times H.$$
This cycle has degree divisible by $4$.

Let $\pi$ be an integral
representative of $S^i(r)$. We want to show that $\deg(\pi b_{d-i}(X^2))$
is divisible by $4$.
By dimension reasons $S^i(H)=0$ for $1\le i\le d-1$. Therefore
$S^i(r_F)=0$ for such $i$'s. Thus, there exists a cycle $\gamma$ over $F$
such that $\pi_F=2\gamma$. It remains to show that
$\deg(\gamma b_{d-i}(X^2))$ is divisible by $2$.

We set $X=X_0=X_1$
and write $X^2=X_0\times X_1$.
We have $$b_{d-i}(X^2)=\sum_{j+r=d-i}b_j(X_0)b_r(X_1).$$
By assumptions $j+r=d-i>0$. Therefore $r>0$ or $j>0$.
In the first case $\deg(\gamma b_j(X_0)b_r(X_1))=
\deg((\pi_0)_*(\gamma b_j(X_0)b_r(X_1)))$.
Therefore it suffices to show that $\deg(\alpha_F\beta)$
is divisible by $2$ for all $\alpha\in\CH^r(X)$, $\beta\in\CH^{d-r}(X_F)$.

Let $\varphi$ be a preimage of $\beta$ under the natural map
$$\CH^{d-r}(X^2)\to\CH^{d-r}(X\times\Spec F).$$ Consider the
cycle $\omega=r(\alpha\times 1)\varphi\in\CH_0(X^2)$.

One has $\omega_F=(H\times 1)(\alpha_F\times 1)\varphi_F+
(1\times H)(\alpha_F\times 1)\varphi_F=
(1\times H)(\alpha_F\times 1)\varphi_F$.
Therefore $\deg\omega=\deg(\alpha_F\beta)=(\pi_0)_*((\alpha_F\times 1)\varphi_F)
=c(r)\deg(\alpha_F\beta)$ is divisible by $2$, since
$X^2$ has no zero-cycles of odd degree.

The case $j>0$ is similar.

Thus $\deg b_d(X)\ne 0$ mod $p^2$. By \cite[Lemma~9.10]{Ro07}
$\dim X=2^{n-1}-1$ for some $n$. Finally, by \cite[Lemma~9.13]{Ro07}
$X$ is a $\nu_{n-1}$-variety.
\end{proof}

Now we return to the settings of Theorem~\ref{binmot}. Define $n$ as the unique number such that
$2^{n-2}-1<d<2^{n-1}$.
Denote by $$\mu\in\Hom(\mathcal{X},\zz_2(d)[2d+1])=\Hom(\mathcal{X},\zz/2(d)[2d+1])$$
(by Lemma~\ref{transfer}) the pull-back
of $\mu'$ under the map of Lemma~\ref{localize}(II). It follows from Lemma~\ref{rostbinary}
that condition (a) of Theorem~\ref{binmot} implies condition (b) with $Y=X$.

The following lemmas were proven by O.~Izhboldin and A.~Vishik
in the case when $X$ is a quadric. Nevertheless their
proofs are general and do not use any specific of quadrics. 
We reproduce them below for
reader's convenience to make the exposition more self-contained.

\begin{lem}[{\cite[Sublemma~6.3]{IzhV00}}]\label{prevlem}
The map
\begin{align}\label{lemma63map}
\Hom(\mathcal{X},\zz_2(c)[e])\to\Hom(\mathcal{X},\zz_2(c+d)[e+2d+1])
\end{align}
induced by the multiplication by $\mu$ is surjective for $e\ge c$. The same holds for
cohomology with $\zz/2$-coefficients.
\end{lem}
\begin{proof}
Consider the morphism from the exact triangle
$$M\to\mathcal{X}\xrightarrow{\mu'}\mathcal{X}(d)[2d+1]$$
to $\zz_2(d+c)[2d+e+1]$.

We get the long exact sequence
\begin{align*}
\Hom(\mathcal{X}(d)[2d+1],\zz_2(d+c)[2d+e+1])\xrightarrow{\mu'^*}
\Hom(\mathcal{X},\zz_2(d+c)[2d+e+1])\\
\to\Hom(M,\zz_2(d+c)[2d+e+1]).
\end{align*}

Notice that $\Hom(\mathcal{X}(d)[2d+1],\zz_2(d+c)[2d+e+1])=
\Hom(\mathcal{X},\zz_2(c)[e])$ and by Lemma~\ref{izhvishik} the map
$\mu'^*$ coincides with the multiplication by $\mu$. Moreover,
\begin{multline*}
\Hom(M,\zz_2(d+c)[2d+e+1])\\
=\Hom(M\{\dim X-d\},\zz_2((\dim X-d)+d+c)[2(\dim X-d)+2d+e+1])=0
\end{multline*}
by Lemma~\ref{zerohom},
since $M\{\dim X-d\}$ is a direct summand of $X$ and $$2(\dim X-d)+2d+e+1-((\dim X-d)+d+c)=\dim X+e-c+1>\dim X.$$
Therefore the multiplication by $\mu$ is surjective for $e\ge c$.

The case of $\zz/2$-coefficients follows from the integral case and from exact triangle~\eqref{zp}.
\end{proof}

Let $\widetilde{\mathcal{X}}=\Cone(\mathcal{X}\to\Spec k)$.

\begin{lem}[{\cite[Sublemma~6.6]{IzhV00}}]\label{milnorinj}
For any $i=0,1,\ldots,n-2$ the Milnor operation
$$Q_i\colon\widetilde H^{e,c}(\widetilde{\mathcal{X}},\zz/2)\to
\widetilde H^{e+2^{i+1}-1,c+2^i-1}(\widetilde{\mathcal{X}},\zz/2)$$
is injective if $e-c=d+2+2^i$.
\end{lem}
\begin{proof}
Let $v\in\widetilde H^{e,c}(\widetilde{\mathcal{X}},\zz/2)$ with $Q_i(v)=0$.
We want to show that then $v=0$.

Since $Q_i(v)=0$, by Lemma~\ref{margolislemma} there exists
$$t\in\widetilde H^{e-2^{i+1}+1,c-2^i+1}(\widetilde{\mathcal{X}},\zz/2)$$
such that $v=Q_i(t)$.

On the other hand, the exact triangle of pointed simplicial schemes
$$\widetilde{\mathcal{X}}[-1]\overset{\tau}{\to}\mathcal{X}_+\to\Spec k_+$$
induces with help of Lemma~\ref{zerohom} an isomorphism
$$\tau^*\colon H^{e-2^{i+1},c-2^i+1}(\mathcal{X},\zz/2)\to
\widetilde H^{e-2^{i+1}+1,c-2^i+1}(\widetilde{\mathcal{X}},\zz/2).$$

Therefore there exists $$w\in\Hom(\mathcal{X},\zz/2(c-2^i+1)[e-2^{i+1}])$$
with $\tau^*(w)=t$. Moreover, by Lemma~\ref{prevlem} there exists
$$u\in\Hom(\mathcal{X},\zz/2(c-2^i+1-d)[e-2^{i+1}-2d-1])$$
which maps to $w$ under the homomorphism~\eqref{lemma63map}.

By Lemma~\ref{milnormulti}(III) we have $Q_i(\mu\cdot u)=Q_i(\mu)\cdot u+\mu\cdot
Q_i(u)+\sum\xi^{n_j}\varphi_j(\mu)\psi_j(u)$,
where $\xi$ is the class of $-1$ in $\Hom(\Spec k,\zz/2(1)[1])$,
$n_j>0$ and $\varphi_j$, $\psi_j$ are some homogeneous cohomological operations
of some bidegrees $(*)[*']$ with $*'>2*\ge 0$.

Since $f:=e-2^{i+1}-2d-1=c-2^i+1-d$ by assumptions, Lemma~\ref{beilinson2}
implies an isomorphism $\pr^*\colon\Hom(\Spec k,\zz/2(f)[f])\to\Hom(\mathcal{X},\zz/2(f)[f])$, where
$\pr\colon\mathcal{X}\to\Spec k$ is the structural map.
Therefore there exists $$u_0\in\Hom(\Spec k,\zz/2(f)[f])$$ with $u=\pr^*(u_0)$.

On the other hand, since by Lemma~\ref{zerohom} $\Hom(\Spec k,\zz/2(h)[g])=0$ for all $g>h$,
we have $Q_i(u_0)=0$ and $\psi_j(u_0)=0$.
Since the Milnor operations commute with pull-backs, we also have
$Q_i(u)=0$ and $\psi_j(u)=0$.

Summarizing we obtain:
$$v=Q_i(t)=Q_i(\tau^*(w))=\tau^*(Q_i(w))=\tau^*(Q_i(\mu\cdot u))=
\tau^*(Q_i(\mu)\cdot u)=0,$$
where the last equality follows from Lemma~\ref{milnortriv}.
\end{proof}

\begin{lem}[{\cite[Sublemma~6.7]{IzhV00}}]\label{finallemma}
We have $d=2^{n-1}-1$ and Lemma~\ref{lll} holds for $p=2$, i.e., there exists
$\delta\in\Hom(\mathcal{X},\zz/2(n-1)[n])$
such that $\mu=Q_{n-2}\ldots Q_1Q_0(\delta)$.
\end{lem}
\begin{proof}
Using Lemma~\ref{zerohom} we can
identify $\mu$ with its image in the reduced motivic cohomology
$\widetilde H^{2d+2,d}(\widetilde{\mathcal{X}},\zz/p)$. We denote this image as $\eta=\eta_{-1}$.

Assume by induction that we have constructed an element $\eta_m$,
$-1\le m\le n-3$, such that $\eta=Q_m\ldots Q_1Q_0(\eta_m)$.
We want to show that there exists $\eta_{m+1}$ with $$\eta_m=Q_{m+1}(\eta_{m+1}).$$

By Lemma~\ref{margolislemma} it suffices to show that
$v:=Q_{m+1}(\eta_m)=0$. We have
\begin{align*}
Q_mQ_{m-1}\ldots Q_0(v)=Q_mQ_{m-1}\ldots Q_0Q_{m+1}(\eta_m)\\
=Q_{m+1}Q_mQ_{m-1}\ldots Q_0(\eta_m)=Q_{m+1}(\eta)=0
\end{align*}
by Lemma~\ref{milnortriv}.

A straightforward computation shows that for any $0\le t\le m$
the element
$$Q_{t-1}\ldots Q_0(v)\in \widetilde H^{e,c}(\widetilde{\mathcal{X}},\zz/2)$$
with $e-c=d+2+2^t$. By Lemma~\ref{milnorinj} $Q_t$
is injective on $\widetilde H^{e,c}(\widetilde{\mathcal{X}},\zz/2)$
for such $c$ and $e$. Therefore the equality $Q_m Q_{m-1}\ldots Q_0(v)=0$ implies $v=0$.

Thus, in this way we can construct an element
$$\widetilde\delta\in \widetilde H^{2d-2^n+n+3,d-2^{n-1}+n}(\widetilde{\mathcal{X}},\zz/2)$$
such that $Q_{n-2}\ldots Q_1Q_0(\widetilde\delta)=\mu$ (recall that we identify $\mu$ with its image
in $\widetilde H^{2d+2,d}(\widetilde{\mathcal{X}},\zz/2)$). By the assumptions the variety $X$ has no $0$-cycles
of odd degree. Therefore by Lemma~\ref{defeta} $\mu\ne 0$. Therefore $\widetilde\delta\ne 0$. On the other hand,
if $d<2^{n-1}-1$, then $2d-2^n+n+2\le d-2^{n-1}+n$. Therefore by \cite[Cor.~6.9]{Vo01} and \cite[Prop.~2.7]{Vo96}
$$\widetilde H^{2d-2^n+n+3,d-2^{n-1}+n}(\widetilde{\mathcal{X}},\zz/2)=0.$$ Hence $d=2^{n-1}-1$ and
$\widetilde\delta\in\widetilde H^{n+1,n-1}(\widetilde{\mathcal{X}},\zz/2)$ and can be identified with an element
$\delta\in\Hom(\mathcal{X},\zz/2(n-1)[n])$.

Alternatively, in case~(a) of Theorem~\ref{binmot} the claim that $d=2^{n-1}-1$ follows from Lemma~\ref{rostbinary}.
\end{proof}

\noindent
{\it Proof of Theorem~\ref{binmot}:}
The first part of the Theorem follows from Lemma~\ref{finallemma}.

The proof of the second assertion repeats word by word the proof
of Theorem~\ref{mainthm} given at the end of the previous section
with Lemma~\ref{lll} replaced by Lemma~\ref{finallemma}.

\medskip

\noindent
{\it Proof of Theorem~\ref{mainthm} for $p=2$:}
This is a particular case of Theorem~\ref{binmot}, since Lemma~\ref{triang} gives us exact triangle~\eqref{eqtr61}.

\section{Compression of varieties and some supplementary results}\label{seccompr}
The goal of this section is to provide a certain compression algorithm
for smooth projective varieties. We assume that $\mathrm{char}\,k=0$ and fix a prime $p$.

\begin{lem}\label{hironak}
Let $X$ be a smooth projective irreducible variety and $Y$ an irreducible
closed subvariety of $X$ of minimal dimension such that $Y_{k(X)}$
has a zero-cycle of degree coprime to $p$.

Then there exists a smooth projective irreducible variety
$\widetilde Y$ which is birational to $Y$ together with a morphism $\widetilde Y\to Y$ and
such that $\widetilde Y_{k(X)}$ has a zero-cycle of degree coprime to $p$.
\end{lem}
\begin{proof}
Let $U\subset Y$ be the open subvariety of smooth points in $Y$ and
$V=Y\setminus U$ the singular locus. We claim that there is a zero-cycle
of degree coprime to $p$ supported in $U_{k(X)}$.
Indeed, assume the contrary. Then all zero-cycles supported on $U_{k(X)}$ have
degrees divisible by $p$. Since there exists a zero-cycle on $Y_{k(X)}$
of degree coprime to $p$, there is an irreducible component $V'$ of $V$
such that the variety $V'_{k(X)}\subset X_{k(X)}$
has a zero-cycle of degree coprime to $p$.
Since $\dim V'<\dim Y$ and $\dim Y$ is minimal by our assumptions, we come
to a contradiction.

Using the standard algorithm of resolution of singularities due to Hironaka
we can find a smooth projective variety $\widetilde Y$ birational to $Y$ which contains $U$ together with a morphism
$\widetilde Y\to Y$ identical on $U$.
This $\widetilde Y$ has a zero-cycle of degree $1$ mod $p$ over $k(X)$.
\end{proof}

A smooth projective irreducible variety $X$ over $k$ is called {\it generically split},
if the motive of $X_{k(X)}$ is a finite direct sum of twisted Tate motives.

\begin{cor}\label{lemmacompr}
Let $X$ be a generically split smooth projective geometrically irreducible
$k$-variety. Let $Y$ be a closed irreducible subvariety of $X$ of minimal dimension
such that the class $[Y_{k(X)}]$ in $\Ch(X_{k(X)})$ is non-zero.
Then there exists a smooth projective irreducible variety
$\widetilde Y$ birational to $Y$ together with a morphism $\widetilde Y\to Y$ and such that $\widetilde Y_{k(X)}$ has
a zero-cycle of degree $1$ mod $p$.
\end{cor}
\begin{proof}
Since $X$ is generically split, by \cite[Remark~5.6]{KM06} and \cite[Prop.~1.5]{Me08}
there exists a closed subvariety $Z\subset X_{k(X)}$ such that $[Y_{k(X)}]\cdot [Z]$ has
degree coprime to $p$.
Since the product $Y_{k(X)}\cdot Z$ in the Chow ring can be
represented by a cycle on the intersection $Y_{k(X)}\cap Z$,
the variety $Y_{k(X)}\cap Z\subset Y_{k(X)}$ has a zero-cycle
of degree coprime to $p$.

By Lemma~\ref{hironak} it suffices to show that
$Y$ has minimal dimension among all closed irreducible subvarieties of $X$
such that $Y_{k(X)}$ has a zero-cycle of degree coprime to $p$.

The proof of this claim copies the proof of \cite[Theorem~5.8]{KM06}.

Let $Y'\stackrel{in}{\hookrightarrow} X$ be a closed irreducible subvariety of $X$ of minimal dimension such that $Y'_{k(X)}$
has a zero-cycle of degree coprime to $p$. It remains to show that $\dim Y'\ge\dim Y$.

For an arbitrary variety $X$ over $k$ we write (following \cite{KM06}) $\Ch(\overline X)$ for the colimit of
$\Ch(X_L)$ over all field extensions $L/k$ and we write $\overline{\Ch}(X)$
for the image of the restriction map $\Ch(X)\to\Ch(\overline X)$. Notice that if $X$ is generically split,
then the groups $\Ch(\overline X)$ and $\Ch(X_{k(X)})$ can be canonically identified.

It suffices to show that $\overline{\Ch}_i(X)$ is non-zero for some $i\le\dim Y'$.
Notice that exactly as in \cite[Remark~5.6 et seq.]{KM06} we have $\Ch_0(\overline X)=\zz/p$.
By \cite[Corollary~5.4]{KM06} it suffices to show that the push-forward
$$(in\times\id_X)_*\colon\overline{\Ch}(Y'\times X)\to\overline{\Ch}(X\times X)$$
is non-zero.

Denote $F=k(X)$. By the assumptions there exists a $p$-coprime field extension $K/F$ such that
$k(Y')\hookrightarrow K$.
Let $W$ be the closure of the image of the morphism $K\to Y'\times X$.
Then the cycle $(in\times\id_X)_*([W])$ is non-zero, since for the second projection
$\pi_1\colon X^2\to X$ we have
$$(\pi_1)_*(in\times\id_X)_*([W])=[K:k(X)]\cdot 1\ne 0.$$
\end{proof}

\begin{dfn}\label{uppermot}
Let $X$ be a smooth projective variety over $k$, let $p$ be a prime, and
$R=(X,\pi)$ be a direct summand of the motive of $X$ with $\zz/p$-coefficients.
We say that $R$ supports $0$-cycles of $X$, if $\Ch^0(R)\ne 0$.
\end{dfn}

\begin{dfn}\label{gensplmot}
Let $X$ be a smooth projective irreducible variety over $k$.
A motive $M=(X,\pi)$ is called a generically split direct summand of $X$, if
$M_{k(X)}$ is a direct sum of twisted Tate motives over $k(X)$.
\end{dfn}

\begin{lem}\label{commonmotive}
Let $X$ be a twisted flag variety over $k$ of inner type and $p$ be a prime number. Let $R$ be an indecomposable generically split
direct summand of the motive of $X$ with $\zz/p$-coefficients supporting $0$-cycles. Let $Y$ be a smooth projective
irreducible variety over $k$. Assume that $X$ has a zero-cycle of degree $1$ mod $p$ over $k(Y)$
and $Y$ has a zero-cycle of degree $1$ mod $p$ over $k(X)$.
Then $R$ is an indecomposable direct summand of $Y$ and $R$ splits over $k(Y)$.
\end{lem}
\begin{proof}
We denote by $pt_X$ (resp. $pt_Y$) a zero-cycle of degree $1$ mod $p$ on $X$ over $k(Y)$ (resp. on $Y$ over $k(X)$).
Since $X$ has a zero-cycle of degree $1$ mod $p$ over $k(Y)$,
there exists by the localization sequence (generic point diagram)
a cycle
$\varphi_1\in\Ch_{\dim X}(X\times Y)$ such that
$\varphi_1\colon pt_X\mapsto pt_Y$ over $k(Y)$.

By symmetry there exists a cycle $\psi_1\in\Ch_{\dim Y}(Y\times X)$
such that $\psi_1\colon pt_Y\mapsto pt_X$ over $k(X)$.

Let $r$ be a projector defining $R$, i.e., $R=(X,r)$, and
consider the cycles $\varphi_2=\varphi_1\circ r$ and $\psi_2=r\circ\psi_1$.
 
Consider now $\End(R_{k(X)})$. Since $R$ is generically split summand of $X$, this is a finite group.
Therefore by the Fitting lemma some power, say $m\ge 1$, of the cycle $(\psi_2\circ\varphi_2)_{k(X)}$
is an idempotent. The cycle $(\psi_2\circ\varphi_2)^m\in\End(R)$ and is non-zero,
since $\psi_2\circ\varphi_2\colon pt_X\mapsto pt_X$ over $k(X)$. Therefore, since $R$ is indecomposable,
the cycle $(\psi_2\circ\varphi_2)^m_{k(X)}$ equals $r_{k(X)}$.

By \cite[Section~8]{CGM05} Rost nilpotence holds for $R$, i.e., for any field extension $K/k$ the kernel
of the natural map $\End(R)\to\End(R_K)$ consists of nilpotent correspondences.
Therefore $(\psi_2\circ\varphi_2)^m=r+n$, where $n$ is some nilpotent element
in $\End(R)$. Since $n$ is nilpotent, $r+n$ is invertible and
$(r+n)^{-1}\circ(\psi_2\circ\varphi_2)^m=r$. Define now
$\psi_3=(r+n)^{-1}\circ\psi_2$ and $\varphi_3=\varphi_2\circ(\psi_2\circ\varphi_2)^{m-1}$. Then $\psi_3\circ\varphi_3=r$, and
therefore $\varphi_3\circ\psi_3$ is a projector on $Y$ and
$(Y,\varphi_3\circ\psi_3)\simeq(X,r)=R$ with mutually inverse isomorphisms
$\varphi_3$ and $\psi_3$.

Over $k(Y)$ the variety $X$ has a zero-cycle of degree coprime to $p$.
Let $L$ be a finite field extension of $k(Y)$ of degree coprime to $p$ such that $X_L$ has a rational point.
By \cite[Thm.~21.20(ii)]{Bor} the variety $X_L$ is rational and, hence, the field extension $L(X_L)/L$ is purely
transcendental. Since the motive $R$ splits over $k(X)$, it also splits over $L(X_L)$ and, hence, over $L$.
Therefore, since the degree $[L:k(Y)]$ is coprime to $p$, the motive $R$ splits already over $k(Y)$.
\end{proof}

\begin{lem}\label{lift}
Let $X$ be a twisted flag variety over a field $k$ and $p$ be a prime number.
Assume that we are given a motivic decomposition of $X$ with $\zz/p$-coefficients of the form
$$X\simeq\bigoplus_{i\in I} R\{i\}$$
for some indecomposable motive $R$ and some multiset $I$ of non-negative indices containing $0$.

Then there is a motivic decomposition of $X$ with $\zz_p$-coefficients of the form
$$X\simeq\bigoplus_{i\in I} \widetilde R\{i\}$$
where the motive $\widetilde R$ is indecomposable and $R\simeq \widetilde R\mod p$.
\end{lem}
\begin{proof}
By \cite[Theorem~4.3]{SZh13} there is a motivic decomposition of $X$ with $\zz_p$-coefficients
of the form
$$X\simeq\bigoplus_{i\in I}\widetilde R_i$$ such that for all $i\in I$ the motives $\widetilde R_i$
are indecomposable and $R\{i\}\simeq\widetilde R_i\mod p$. Let $\widetilde R$ denote one of the motives $\widetilde R_j$, $j\in I$, such that $R\simeq \widetilde R_j\mod p$.
The proof that $\widetilde R_i\simeq\widetilde R\{i\}$ for all $i\in I$ is similar to the proof
of \cite[Theorem~4.3]{SZh13}, where instead of lifting of projectors one applies the same procedure to lift
the isomorphisms $$(\widetilde R_i\mod p)\simeq (\widetilde R\{i\}\mod p).$$
\end{proof}

\section{Serre's question about $\E_8$}\label{sec8}
Let $G$ be a simple linear algebraic group of inner type over $k$
with $\mathrm{char}\,k=0$. We briefly sketch
the definitions of two invariants of $G$:
the $J$-invariant and the Rost invariant.
See \cite{PSZ} and \cite[\S31]{Inv} for formal definitions.

\begin{ntt}[$J$-invariant]
Let $p$ be a prime number and
$R$ be a motive which is a finite direct sum of twisted Tate
motives. The polynomial $$P(R,t)=\sum_{i\ge 0}a_it^i\in\zz[t],$$
where $a_i=\dim\Ch^i(R)$ is called the Poincar\'e polynomial of $R$ mod $p$.

Informally speaking, the $J$-invariant of a simple group $G$ is a discrete invariant
which measures the motivic decomposition of the variety of Borel
subgroups of $G$.

Consider the motive of the variety of complete
$G$-flags (variety of Borel subgroups) with $\zz/p$- or
$\zz_{(p)}$-coefficients. It turns out that it is isomorphic
to a direct sum of shifted copies of the same indecomposable generically direct summand $R=R_p(G)$ such that over any
splitting field $K$ of $G$ the Poincar\'e polynomial of $R_K$
equals
\begin{equation}\label{jinva}
P(R_K,t)=\prod_{i=1}^r\frac{t^{d_ip^{j_i}}-1}{t^{d_i}-1}
\end{equation}
for some 
given numbers $r$ and $d_1\le d_2\le\ldots\le d_r$ which depend only on the Dynkin type of $G$
and some $r$-tuple $(j_1,\ldots,j_r)$ of non-negative integers
(see \cite[Definition~4.5 and Theorem~4.9]{PSZ}).

The tuple $J_p(G)=(j_1,\ldots,j_r)$ arising from this decomposition is called
the {\it $J$-invariant} of $G$ modulo $p$. We leave the problem
of correctness of such a definition aside and refer to \cite{PSZ}
and \cite[Section~3]{QSZ} for more details.

An important application of the $J$-invariant which is essentially
used in the proof of Serre's problem described below is the classification
of generically split twisted flag varieties given in \cite[Theorems~5.5 and 5.7]{PS}
and \cite[Thm.~3.3]{PS11}.

Notice also that formula~\eqref{jinva} implies that
the variety of Borel subgroups of $G$
has a binary direct summand if $J_2(G)=(0,\ldots,0,1,0,\ldots,0)$.
In Lemma~\ref{lemmae8} we'll provide a concrete example.

\begin{ex}\label{eee}
If $G$ has type $\E_8$ and $p=2$, then $r=4$, $d_1=3$, $d_2=5$, $d_3=9$, and $d_4=15$.
Moreover, $j_4=0$ or $1$, and if $j_1=0$ then $j_2=j_3=0$ (see \cite[\S4, Table]{PSZ}).
\end{ex}

\end{ntt}

\begin{ntt}[Rost invariant]
Let $G$ be a simply connected simple algebraic group over $k$.
The Rost invariant of $G$ is an invariant of $G$-torsors which takes
values in $H_{et}^3(k,\qq/\zz(2))$.

Formally: consider the abelian group
$\mathrm{Inv}^3(G,\qq/\zz(2))$ of natural transformations of functors
$H_{et}^1(-,G)\to H_{et}^3(-,\qq/\zz(2))$ defined on the category of field extensions
of $k$.

It turns out that it is a finite cyclic group with a canonical generator
called the {\it Rost invariant} of $G$. If $G$ is adjoint and simply connected,
then one can identify $H_{et}^1(k,G)$ with the (pointed) set of isomorphism classes
of the twisted forms of $G$ over $k$. In particular, in this case one can
associate with each twisted form of $G$ an element in $H_{et}^3(k,\qq/\zz(2))$
(the Rost invariant). More generally, if $G$ is a simple group with trivial Tits algebras,
then one may speak about its Rost invariant, see \cite[Section~2]{GP07}.

\begin{ex}
Let $G$ be a group of type $\E_8$. It is known that it is adjoint
and simply connected, and the Rost invariant of $G$
takes values in
$$H_{et}^3(k,\zz/4)\oplus H_{et}^3(k,\zz/3)\oplus H_{et}^3(k,\mu_5^{\otimes 2}),$$
i.e., consists of a mod-$4$ (or mod-$2$), mod-$3$, and mod-$5$ component (see \cite[\S9--16 and App.~A]{GMS}
and references there). Notice that $4\cdot 3\cdot 5=60$ is the Dynkin index of $\E_8$.
If the group $G$ splits over a field extension of degree coprime
to $l$, then the $l$-component of its Rost invariant is zero $(l\in\zz)$.
\end{ex}
\end{ntt}

\begin{lem}\label{lemmae8}
Let $G$ be a group of type $\E_8$ and $X$ the variety of Borel
subgroups of $G$. If the even component of the Rost
invariant of $G$ is trivial, then the motive of $X$ with
$\zz/2$- (or $\zz_{2}$-)coefficients is a direct sum of
twisted copies of a motive $R$ whose Poincar\'e polynomial over any splitting field of $G$
equals $1+t^{15}\in\zz[t]$.
\end{lem}
\begin{proof}
Since we are talking about the even component of the Rost invariant and
about the motives with $\zz/2$-coefficients we may assume that the base
field $k$ does not have any proper finite field extension of odd degree.

Let $X_i$ be the projective $G$-homogeneous variety of maximal parabolic
subgroups of $G$ of type $i$, $i=1,\ldots,8$. If the Rost invariant of $G$ is trivial,
then so is the Rost invariant of $G_{k(X_i)}$ and of its semisimple anisotropic kernel.
On the other hand, since by the Tits classification \cite{Ti66}
the semisimple anisotropic kernel of $G_{k(X_i)}$ is of type $\D_7$, $\E_7$, $\D_6$, $\E_6$,
is simple of rank smaller than $6$, or is trivial,
the group $G_{k(X_i)}$ must be split (see e.g. \cite[Theorem~0.5]{Ga01}).
Thus, $X_i$ is a generically split variety for all $i=1,\ldots,8$.

The classification of generically split varieties of type $\E_8$
(see \cite[Theorem~5.7(7)]{PS}) and the classification of the
$J$-invariants for groups of type $\E_8$ (see \cite[\S4, Table]{PSZ}
or Example~\ref{eee})
immediately imply that the $J$-invariant of $G$ modulo $2$ equals
either $J_2(G)=(0,0,0,1)$ or $(0,0,0,0)$.

In the first case we are done by formula~\eqref{jinva}. In the second
case the variety of Borel subgroups is a direct sum of twisted Tate
motives and the present lemma is trivial.

Finally, by Lemma~\ref{lift}
one can lift a motivic decomposition of $X$ with
$\zz/2$-coefficients to a motivic decomposition of $X$ with $\zz_{2}$-coefficients,
and the conclusion of the present lemma holds with $\zz_{2}$-coefficients as well.
\end{proof}

\begin{lem}\label{e8tr}
In the notation of Lemma~\ref{lemmae8} let $\mathcal{X}$ be the motive of the standard simplicial scheme associated with
the variety $X$. Then for the motive $R$ there is an exact triangle of the form
$$\mathcal{X}\{15\}\to R\to \mathcal{X}\to\mathcal{X}\{15\}[1].$$
\end{lem}
\begin{proof}
The proof is the same as the proof of \cite[Thm.~4.4]{Vo01} with Theorem~4.3 of \cite{Vo01} replaced by
Lemma~\ref{lemmae8} above (cf. Lemma~\ref{triang}).
Notice that Theorem~4.4 in \cite{Vo01} is formulated for norm quadrics, but the proof does not use any specific of
norm quadrics, only the existence of a binary direct summand.
\end{proof}

\begin{thm}\label{corre8e8}
Let $k$ be a field with $\mathrm{char}\,k=0$.
Let $G$ be a group of type $\E_8$ over $k$ whose even component of the Rost invariant is trivial.
Then there exists a functorial invariant $u\in H_{et}^5(k,\zz/2)$ of $G$ such that
for any field extension $K/k$ the invariant $u_K=0$ iff $G$ splits over
a field extension of $K$ of odd degree.
\end{thm}
\begin{proof}
Let $X$ be the variety of Borel subgroups of $G$. It is generically
split, since over $k(X)$ the group $G$ splits. By Lemma~\ref{lemmae8}
the motive of $X$ contains a binary direct summand $R$ which supports
zero-cycles of $X$.

Without loss of generality we may assume that $X$ has no zero-cycles of odd degree over $k$.
Consider the image $\overline{\Ch}(X)$
of the natural map $\Ch(X)\to\Ch(X_{k(X)})$.
By \cite[Proposition~6.1]{PSZ} and \cite[Theorem~5.8]{KM06}
$$\min\{i\mid\overline{\Ch}_i(X)\ne 0\}=15.$$

By Corollary~\ref{lemmacompr} there exists a smooth projective irreducible
variety $\widetilde Y$ of dimension $15$
birational to some closed
subvariety $Y$ of $X$ together with a morphism $\widetilde Y\to Y\hookrightarrow X$ such that $\widetilde Y_{k(X)}$ has a
zero-cycle of odd degree.
The converse obviously holds, i.e., $X_{k(\widetilde Y)}$ has
a zero-cycle of odd degree.

Therefore by Lemma~\ref{commonmotive} the variety $\widetilde Y$
has the same direct summand $R$. It follows from Lemma~\ref{rostbinary}
that $\widetilde Y$ is a $\nu_4$-variety. We are done by Theorem~\ref{binmot} and Lemma~\ref{e8tr}.
\end{proof}

The next theorem gives a positive answer to Serre's question described in the Introduction.

\begin{thm}\label{maintheorem}
Let $G$ be a group of type $\E_8$ over $\qq$ such that $G_\rr$ is
a compact Lie group, let $K/\qq$ be a field extension,
and $q=\langle\!\langle -1,-1,-1,-1,-1\rangle\!\rangle$
a $5$-fold Pfister form.

If $G_K$ is split, then $q_K$ is hyperbolic.
\end{thm}
\begin{proof}
The compact real group of type $\E_8$ has Rost invariant zero (see
\cite[13.4]{Ga08}), so we may speak of its invariant
$u\in H_{et}^5(\rr,\zz/2)=\{0,(-1)^5\}$ constructed in Theorem~\ref{corre8e8}.

If $u=0$, then by Theorem~\ref{corre8e8}
this compact group splits over a field extension of odd degree.
Since $\rr$ has only one field extension of odd degree, we come to
a contradiction. Thus, $u=(-1)^5$.

Let now $G$ be a group of type $\E_8$ defined over $\qq$ that becomes
compact $\E_8$ over $\rr$. 
Set $F=\qq(\sqrt{-1})$. Then by \cite[Ch.~II, \S4.4, Prop.~13]{Se97} $H_{et}^d(F,\zz/p^m(d-1))=0$ for all prime
numbers $p$, all $m\ge 1$ and all $d\ge 3$. By restriction--corestriction argument, $H_{et}^d(\qq,\zz/p^m(d-1))=0$
if $p$ is odd, and $H_{et}^d(\qq,\zz/2^m(d-1))$ is $2$-torsion for all $m\ge 1$ and all $d\ge 3$.

Therefore by \cite[App.~A]{GMS} the Rost invariant takes values in $$H_{et}^3(\qq,\qq/\zz(2))=H_{et}^3(\qq,\zz/2).$$
Since $H_{et}^d(F,\zz/2)=0$ for $d\ge 3$, \cite[Cor.~30.12(1)]{Inv} gives that the multiplication by $(-1)$
is an injection from $H_{et}^d(\qq,\zz/2)\to H_{et}^{d+1}(\qq,\zz/2)$. It follows now from \cite[Satz~3]{Ar75} that
$H_{et}^d(\qq,\zz/2)$ injects into $H_{et}^d(\rr,\zz/2)$.

Therefore, $G$ has Rost invariant zero over $\qq$, and
again we may speak of the invariant $u\in H_{et}^5(\qq,\zz/2)$ of the group $G$.
But restriction identifies $H_{et}^5(\qq,\zz/2)$ with $H_{et}^5(\rr,\zz/2)$. So $u=(-1)^5$.

Finally, if $G_K$ splits for some field extension $K/\qq$, then $u_K=0$
by Theorem~\ref{corre8e8}. We are done.
\end{proof}

\bibliographystyle{chicago}

\medskip

\medskip

\noindent
\sc{Nikita Semenov\\
Johannes Gutenberg-Universit\"at Mainz, Institut f\"ur Mathematik, Staudingerweg 9,
D-55128 Mainz, Germany}\\
{\tt semenov@uni-mainz.de}

\end{document}